\documentclass[a4paper, reqno, 11pt]{amsart}

\usepackage[english]{babel}
\usepackage{amsmath}
\usepackage{amssymb}
\usepackage{amsthm}
\usepackage{enumerate}
\usepackage{ifthen}
\usepackage{bbm}
\usepackage{color}
\provideboolean{shownotes} 
\setboolean{shownotes}{true}
\usepackage{hyperref}
\usepackage{graphicx}
\usepackage{pgfplots}
\usepackage{tikz,tikz-3dplot} 
\usepackage{mathtools}
\usepackage{comment}
\usepackage{float}
\usepackage{geometry}
\newcommand{\margnote}[1]{
\ifthenelse{\boolean{shownotes}}%
{\marginpar{\raggedright\tiny\texttt{#1}}}%
{}%
}

\newcommand{\hole}[1]{
\ifthenelse{\boolean{shownotes}}%
{\begin{center} \fbox{ \rule {.25cm}{0cm}
\rule[-.1cm]{0cm}{.4cm} \parbox{.85\textwidth}{\begin{center}
\texttt{#1}\end{center}} \rule {.25cm}{0cm}}\end{center}}
{}
}
\usepackage{mathrsfs}
\newtheorem{thm}{Theorem}[section]

\newtheorem{prop}[thm]{Proposition}
\newtheorem{lem}[thm]{Lemma}

\newtheorem{rem}[thm]{Remark}

\newtheorem{defn}[thm]{Definition}
\newtheorem*{mainthm*}{Theorem 1}
\newtheorem*{mainthmdue*}{Theorem 2}

\newcommand{\e}{\varepsilon}		       
\newcommand{\R}{\mathbb{R}}

\newcommand{\T}{\mathbb{T}}

\newcommand{\N}{\mathbb{N}}

\newcommand{\Z}{\mathbb{Z}}
\newcommand{\dive}{\mathop{\mathrm {div}}}
\newcommand{\curl}{\mathop{\mathrm {curl}}}

\newcommand{\de}{\mathrm{d}}

\numberwithin{equation}{section}

\subjclass[]{}


\begin{document}

\title[Global smooth solutions and vortex reconnection for N-S in $\R^3$]{Localization of Beltrami fields: global smooth solutions and vortex reconnection for the Navier-Stokes equations}

\author[G. Ciampa]{Gennaro Ciampa}
\address[G.\ Ciampa]{DISIM - Dipartimento di Ingegneria e Scienze dell'Informazione e Matematica\\ Universit\`a  degli Studi dell'Aquila \\Via Vetoio \\ 67100 L'Aquila \\ Italy}
\email[]{\href{gciampa@}{gennaro.ciampa@univaq.it}}

\author[R. Luc\`a]{Renato Luc\`a}
\address[R. Luc\`a]{BCAM - Basque Center for Applied Mathematics, Alameda de Mazarredo 14, E48009 Bilbao, Basque Country - Spain and Ikerbasque, Basque Foundation for Science, 48011 Bilbao, Basque Country - Spain.}
\email[]{\href{rluca@}{rluca@bcamath.org}}

\begin{abstract}
We introduce a class of divergence-free vector fields on $\R^3$ obtained after a suitable localization of {\em Beltrami fields}. First, we use them as initial data to construct unique global smooth solutions of the three dimensional Navier-Stokes equations. The relevant fact here is that these initial data can be chosen to be large in any critical space for the Navier--Stokes problem, however they satisfy the nonlinear smallness assumption introduced in \cite{ChG09}. As a further application of the method, we use these vector fields to provide analytical example of vortex-reconnection for the three-dimensional Navier-Stokes equations on $\R^3$. To do so, we exploit the ideas developed in \cite{ELP} but differently from this latter we cannot rely on the non-trivial homotopy of the three-dimensional torus. To overcome this obstacle we use a different topological invariant, i.e. the number of hyperbolic zeros of the vorticity field.  
\end{abstract}

\maketitle
%\tableofcontents

\section{Introduction}
The Cauchy initial value problem for the three-dimensional incompressible Navier-Stokes equations is given by
\begin{equation}\label{eq:ns}\tag{NS}
\begin{cases}
\partial_t u+(u\cdot \nabla)u+\nabla p=\nu\Delta u,\\
\dive u=0,\\
u(0,\cdot)=u_0,
\end{cases}
\end{equation}
where $u:[0,T)\times \R^3 \to \R^3$ identifies the velocity of the fluid, $p:[0,T)\times\R^3 \to \R$ denotes the pressure, and $\nu> 0$ is the kinematic viscosity. The initial conditions are prescribed at time $t=0$.
The system describes the evolution of a homogeneous incompressible fluid starting from a given configuration $u_0$. It is well-known that the equations \eqref{eq:ns} can be reformulated in terms of the vorticity $\omega:=\curl u$, as follows
\begin{equation}\label{eq:ns-v}\tag{V-NS}
\begin{cases}
\partial_t \omega+(u\cdot \nabla)\omega=(\omega\cdot \nabla)u+\nu\Delta \omega,\\
u=K*\omega,\\
\omega(0,\cdot)=\curl u_0,
\end{cases}
\end{equation}
where $K$ denotes the three-dimensional Biot-Savart operator, which is a $3\times 3$ matrix operator defined as follows
\begin{equation}
K(x)h:=\frac{1}{4\pi}\frac{x\wedge h}{|x|^3},
\end{equation}
where $\wedge$ denotes the vector product in $\R^3$. Understanding the well-posedness of the Navier-Stokes equations is one of the most challenging problems in mathematics. Our aim here is twofold. First, we want to construct a class of initial data with arbitrary large scaling invariant norms, so that the equations are classically well-posed. Secondly, we want to construct analytic examples of {\em vortex reconnection} for solutions defined on the whole space $\R^3$. We now analyze these two problems in more detail.

\subsection{Global smooth solutions} In the seminal work \cite{Leray}, Leray proved the existence of weak solutions of \eqref{eq:ns} on $\R^3$: given any $T>0$ and any divergence-free vector field $u_0\in L^2(\R^3)$, there exists a divergence-free vector field $u\in L^\infty((0,T);L^2(\R^3))\cap L^2((0,T);\dot{H}^1(\R^3))$ that solves \eqref{eq:ns} in the sense of distributions and which additionally satisfies the energy inequality
\begin{equation}
\|u(t,\cdot)\|_{L^2}^2+2\nu\int_0^t\|\nabla u(s,\cdot)\|_{L^2}^2\de s\leq \|u_0\|_{L^2}^2.
\end{equation}
These solutions are commonly known as {\em Leray-Hopf} weak solutions and whether they are unique or not is still an open problem. Some recent progress toward non-uniqueness has recently been established in \cite{BV} for weak solutions and in \cite{ABC} for Leray-Hopf solutions of the forced system.

Here we are interested in global smooth solutions. The existence of local-in-time smooth solutions can be obtained via a standard fixed point argument in several functional spaces. For instance, if we are interested in the Sobolev scale, it turns out that one has local solutions in $H^s(\R^3)$ for~$s>1/2$, with a local existence time that depends on the $H^s$ norm of the initial datum. 
If we are looking for global smooth solutions, a prominent role is played by critical spaces, that is, invariants under the scaling of the Navier-Stokes equations. 
The scaling of the Navier-Stokes equations is the following: let $(u,p)$ a solution of \eqref{eq:ns} and let $\lambda>0$, then the functions
\begin{equation}\label{scaling}
u^\lambda(t,x)=\lambda u(\lambda^2 t,\lambda x),\qquad p^\lambda(t,x)=\lambda^2 p(\lambda^2 t,\lambda x)
\end{equation}
also solves \eqref{eq:ns} with initial datum
\begin{equation}\label{scalingDatum}
u^\lambda_0(x)=\lambda u_0(\lambda x),\qquad p^\lambda_0(x)=\lambda^2 p(\lambda x)
\end{equation}
We say that a  Banach  space is critical  for the initial datum if the relative norm is invariant under \eqref{scalingDatum}. Some of the relevant translation invariant critical Banach spaces for the Navier-Stokes problem are given in the following chain of embeddings %critical spaces, namely
\begin{equation}
\dot{H}^{1/2}(\R^3) \hookrightarrow L^3(\R^3) \hookrightarrow {\dot B^{-1+3/p}_{p,\infty}(\R^3)}_{(p<\infty)} \hookrightarrow BMO^{-1}(\R^3) \hookrightarrow\dot{B}^{-1}_{\infty,\infty}(\R^3),
\end{equation}
and the existence of global-in-time smooth solutions arising from {\it small initial data} in this functional spaces has been established up to $BMO^{-1}(\R^3)$ in \cite{FK, K, GM, CMP, KochTat}. All these results are obtained by looking at fixed points of the functional
\begin{equation}
u=e^{\nu t\Delta}u_0-\int_0^t e^{\nu(t-s)\Delta}\mathbb{P}\dive(u\otimes u)\de s,
\end{equation}
that is an integral reformulation of the differential problem \eqref{eq:ns}, where $e^{\nu t\Delta}$ denotes the heat kernel and $\mathbb{P}$ is the projection on the divergence-free vector fields 
subspace.
It is important to point out that the Besov space $\dot{B}^{-1}_{\infty,\infty}(\R^3)$ is actually the largest translation invariant critical Banach space for the Navier-Stokes equations. However, the Navier-Stokes equations are ill-posed in $\dot{B}^{-1}_{\infty,\infty}(\R^3)$ as shown in \cite{BP}.

The large data problem is still completely open and represent one of the greatest challenge in the field of PDEs. 
Some classes of large solutions can be constructed if one impose some additional symmetry, like for instance
being 2-dimensional  or axisymmetric without swirl (or being a small perturbation of the the former). Here, however, we 
assume a slightly different viewpoint. Looking at the second order Picard iteration (which requires a closer analysis of the nonlinearity of the equation)     
the authors of \cite{ChG06} and \cite{ChG09} provided a smallness condition (of nonlinear type) on the initial datum 
that guarantees the existence of a global-in-time smooth solution. In this way they are able to prove global well-posedness for   initial data with arbitrary large $\dot{B}^{-1}_{\infty,\infty}$ norm, even in absence of a specific symmetry.
Remarkably, other examples of global smooth solutions starting from initial data 
with arbitrary large norm were then constructed in \cite{ChGP}. We also refer to \cite{G, GIP, I, MTL, PRST} where different type of examples are provided.

In this direction, our first main result deals with the construction of further examples of initial data which satisfy the condition introduced in \cite{ChG09} for which the $\dot{B}^{-1}_{\infty,\infty}(\R^3)$ norm can be chosen arbitrary large. To simplify the presentation of our first main result we will set the viscosity $\nu =1$. It is worth mentioning that the precise value of the viscosity does not play an important role in Theorem \ref{MainThm1} (as long as $\nu>0$), while it will be more interesting to keep track of it in 
Theorem \ref{thm:main}.

The nonlinear smallness condition from \cite{ChG09} is the following 
\begin{equation}\label{NonlSmallIntro}
\| \mathbb{P} \, (e^{t\Delta} u_0 \cdot \nabla ) e^{t\Delta} u_0  \|_{E}  \leq \frac{1}{C^*} 
\exp\left(-C^*  \| u_0\|_{\dot{B}^{-1}_{\infty, 2}}^4  \right),
\end{equation}
where $C^*$ is an absolute constant and the $E$ norm is defined below in \eqref{CGSmallCond}. 
Our examples are constructed starting from the following observations. Let consider a Beltrami field $B_\lambda$ with frequency $\lambda$, i.e an eigenvector of the curl operator with eigenvalue $\lambda$
$$
\curl B_\lambda = \lambda B_\lambda.
$$
These vector fields are eigenvectors of the Laplacian $B_\lambda = - \lambda^2 B_\lambda$ and 
steady solutions of the Euler equation % since they satisfy
\begin{equation}
B_\lambda\cdot\nabla B_\lambda+\nabla \left(\frac{|B_\lambda|^2}{2}\right)=0.
\end{equation}
Using these two facts one can see immediately that the vector field
$e^{-t\lambda^2}B_\lambda(x)$ is a solution of the Navier-Stokes equations with 
initial datum $B_\lambda$ \big(and pressure $\frac12e^{-2t\lambda^2} |B_\lambda|^2$\big). 
Since one can rescale these solutions by any arbitrary factor (if $B_\lambda$ is a Beltrami field so it is $\rho B_\lambda$)
we can construct solutions that are large in any Banach space, simply taking the rescaling factor large.
Being eigenfunctions of the Laplacian, Beltrami fields are smooth.
However, the $L^2$ norm of Beltrami fields on $\R^3$ is infinite \cite{CC, Liouville GAFA}.  
Thus, they give rise to infinite energy smooth solutions. Moreover, these solutions are poorly localized, 
since the best decay rate at infinity for a Beltrami field is~$\sim 1/|x|$; see again \cite{CC, Liouville GAFA}. 

Our aim is to localize a Beltrami field in such a way to obtain a new vector field that is smooth and has fast decay at infinity (the localization can be even chosen in such a way that the field is compactly supported) but still retains some of the good properties of the original Beltrami field. In particular, these localized fields will be in $L^2(\R^3)$, thus our solutions will have finite energy. 
 
To do so, we would be tempted to multiply $B_{\lambda}$ against a localized function~$\phi$, however this would not preserve the divergence free nature of the datum. Thus, we consider data of the form
\begin{equation}\label{BlambdaLocIntro}
u_0 := M  \curl(\phi_{L} B_{\lambda}) , \qquad M, L > 0,
\end{equation}
that are divergence free. Here $\phi_{L}$ denotes the function
\begin{equation}\label{RescPhiL}
\phi_{L}(x) := \phi(x/L), 
\end{equation}
namely we are just considering a dilation of a suitable localizing function $\phi \in W^{s,1} \cap W^{s, \infty}$
with $s$ sufficiently large. In practice, we simply ask that $\phi$ is a sufficiently regular
function which decays sufficiently fast together with its derivatives (enough to be integrable). 
Concrete examples 
are, for instance: i) a Gaussian, ii)
a smooth function with compact support, iii) a function of the form $\phi = (1 + |x|^2)^{- \alpha}$ with 
$\alpha  \gg 1$.

We will see that choosing the parameters $M$ and $L$ sufficiently large we have that: 
%To do this, we consider a fast decaying function $\phi_L$ and assume for example that it is identically $1$ on a ball of radius $L>0$. Thus, the vector field
%$$
%u_0=\curl(\phi B_\lambda)=\lambda\phi_L B_\lambda+\nabla\phi_L\wedge B_\lambda,
%$$
%is smooth, divergence-free and it belongs to any Sobolev space $H^s(\R^3)$. In particular, $u_0\in L^2(\R^3)$. Then, we can show that, by choosing parameters appropriately, it is possible to construct such fields that
\begin{itemize}
\item $\|u_0\|_{\dot{B}^{-1}_{\infty,\infty}}$ can be arbitrary large (in fact, $\|u_0\|_{\dot{B}^{-1}_{\infty,\infty}} \simeq M$ see Lemma \ref{SizeLemmaM});
\item $u_0$ satisfies the nonlinear smallness condition \eqref{NonlSmallIntro}; 
\end{itemize}

This leads to our first main result.
\begin{mainthm*}\label{MainThm1}
Let $M > 0$ and $q \in (3, \infty)$. Let $B_{\lambda} \in L^{q}\cap L^{\infty}(\R^3)$ be a Beltrami field of 
frequency~$\lambda \neq 0$ and such that $\|B_\lambda\|_{L^\infty}=1$. Let  %$\|B_\lambda\|_{L^q}=R$, 
 $\phi \in W^{s,1} \cap W^{s, \infty}(\R^3)$ be a positive  function. %a localizing function of order $s$ (see Definition \ref{Def:sLocalizing} below).
We consider the divergence free vector field
\begin{equation}\label{DefW0Bis}
u_0 := M  \curl(\phi_{L} B_{\lambda}) , \quad \phi_{L}(\cdot) := \phi\left( \frac{\cdot}{L} \right), \qquad M, L > 0.
\end{equation}
Then % $u_0\in H^k(\R^3)$ for all $k=0,...,s-1$ and 
for all sufficiently large values of the regularity $s$ and of 
the dilation parameter $L$ the following holds:
\begin{itemize}
\item every scaling invariant norm of $u_0$ is large if $M$ is large. In fact $\|u_0\|_{\dot{B}^{-1}_{\infty,\infty}} \simeq M $;
\item there exists a unique global strong solution $u$ of the Navier-Stokes equation with initial datum $u_0$.
\end{itemize} 
The solution is smooth and has finite energy, in fact satisfies the 
energy identity 
\begin{equation}\label{EnergyId}
\int_{\R^3}|u(t,x)|^2 \, \de x + \int_{0}^{t} \int_{\R^3} |\nabla u(s,x)|^2 \, \de x \de s = 
\int_{\R^3}|u_0(x)|^2 \, \de x .
\end{equation}
\end{mainthm*}

Our assumptions $B_{\lambda} \in  L^{q}\cap L^{\infty}$, for some $q \in  (3, \infty)$ cover a large class of Beltrami fields. For instance, in \cite{Acta} the authors constructed Beltrami fields realizing (with their integral lines) arbitrarily complicated topological structures and with optimal decay at infinity $\sim 1/|x|$. These vector fields belong to $L^{3+\varepsilon}\cap L^{\infty}$, for all $\varepsilon >0$. However, the fact that we cannot take $q = \infty$ leaves out other important classes of Beltrami fields, like for instance the ABC flows. Understanding whether a similar result can be proved for ABC flows is, in our opinion, an interesting open problem. Lastly, we emphasize that the hypothesis $q>3$ is consistent with the Liouville Theorem for Beltrami fields, see \cite{CC, Liouville GAFA}. In particular, if a Beltrami field $B\in L^q(\R^3)$ with $q\in[2,3]$ or $B(x)=o(1/|x|)$ for $|x|\to \infty$, then $B$ must be identically $0$.

\begin{rem}
It is worth remarking that in order to verify that our initial data satisfies the condition \eqref{NonlSmallIntro}, the presence of the Leray projector $\mathbb{P}$ will be crucial.  This is related to the algebraic structure of Beltrami fields and, in particular, to the fact that they are stationary solutions of the Euler equations: $\mathbb{P} \left( (B \cdot \nabla) B \right) =0$.
\end{rem}

\subsection{Vortex reconnection}
Our attention now shifts to the phenomenon of {\em vortex reconnection}: it refers to a change in the topological structure of the vorticity field lines of a solution of \eqref{eq:ns}. We refer to \cite{KT, OC} and reference therein for an overview on the problem from a physical and numerical point of view. We point out that in the first reference the term ``vorticity reconnection" is used instead. From a mathematical point of view, we will say that a solution $u$ of \eqref{eq:ns} shows vortex reconnection if there exist $t_1, t_2$ such that there is no homeomorphism of $\R^3$ mapping the set of the vortex lines of $u(t_1,\cdot)$ into vortex lines of $u(t_2,\cdot)$. A {\em vortex line} of $u$ is an integral line of its vorticity, i.e. a curve $\gamma$ on $\R^3$ which solves the autonomous ODE
\begin{equation}
\dot{\gamma}(s)=\omega(t,\gamma(s)),
\end{equation}
where $t$ is fixed and $s$ represent the parameter that defines the curve. In the case of the Euler equations, i.e. $\nu=0$, a smooth solution $\omega$ of \eqref{eq:ns-v} satisfies the formula
\begin{equation}\label{eq:euler}
\omega(t,\Phi_t(x))=\nabla \Phi_t(x)\omega_0(x),
\end{equation}
where $\Phi_t$ is the fluid flow, i.e. the solution of
$$
\begin{cases}
\frac{\de}{\de t} \Phi_t (x) = u(t,\Phi_t (x)), \\
\Phi_0 (x) = x.
\end{cases}
$$
The formula \eqref{eq:euler} defines the push-forward of $\omega_0$ by $\Phi_t$: if the velocity field $u$ is smooth, $\Phi_t:\R^3\to\R^3$ is a diffeomorphism and the identity \eqref{eq:euler} implies that, at every time $t>0$, the integral lines of $\omega_0$ and $\omega(t,\cdot)$ are diffeomorphic. Thus, as long as the solution remains smooth, there is no reconnection.

On the contrary, in the case of the Navier-Stokes equations, i.e. $\nu>0$, the topology of the vortex lines is expected to change under the fluid evolution, even for regular solutions. The heuristic underlying the phenomenon is that the diffusion allows breaking the topological rigidity. Although experimental and numerical evidences were known, the first analytical examples of this phenomenon were provided by the second author, Alberto Enciso and Daniel Peralta-Salas in \cite{ELP} in the case of periodic smooth solutions, i.e. solutions defined on the three-dimensional torus $\T^3$. We now give a summary of the main ideas of the result in \cite{ELP}: % (see also \cite{L}): 
the strategy of the construction is to consider data for which the velocity at time $t=0$ has the form: 
$$
B_{N_0} + \delta B_{N_1}, \quad 0<  \delta \ll 1, \quad N_0 \gg N_1,
$$ 
where $B_{N_j}$ are high frequency Beltrami fields, which are eigenvectors of the curl operator with eigenvalues (frequency) $N_j$. 
Moreover the fields $B_{N_j}$ satisfy the following properties:
\begin{enumerate}
[$i)$]
\item all the vortex lines of $B_{N_0}$ wind around a certain direction of the torus, in particular all of them are non contractible. This is a robust topological property, in the sense that it is still valid for all sufficiently small regular perturbations of $B_{N_0}$.
\item The field $B_{N_1}$ has some contractible vortex line. This is again topologically robust.
\end{enumerate}
Let us stress the fact that while the Beltrami field $B_{N_0}$ can be constructed explicitly, and relies on the topological properties of the torus $\T^3$, the existence of the field $B_{N_1}$ follows from deep topological results proved in \cite{Annals, Acta, EPT}. In a nutshell, in the results just mentioned, it is shown that given any finite collection of closed curves $\mathcal{S}$, knotted and linked in arbitrary ways, there exists a Beltrami field $B$ with sufficiently high frequency $N$ such that $B$ has a collection of integral lines (contained in a ball of radius $1/N$) which are diffeomorphic to $\mathcal{S}$. The proof of the reconnection is then obtained by choosing the relevant parameters of the construction $\delta, N_0, N_1$ in such a way that the solution will be sufficiently close (in the $C^1$-norm) to $B_{N_0}$ at time $t=0$ and to a suitable rescaled version of the field $B_{N_1}$, at time $t= T >0$. 
The robustness of the topological constraint $i)$, $ii)$ implies that the vortex lines of the solution at time $t=0$ and $t=T$ are not homeomorphic. Thus we must have had vortex reconnections in the intermediate times. \\
\\
Our goal is to provide examples of reconnection in the same spirit of \cite{ELP} but for solutions defined on the full space $\R^3$. In doing so, two main difficulties are encountered. First of all, the homotopy of $\R^3$ is trivial, indeed any curve embedded in $\R^3$ is contractible. Thus we have to identify a different topological constraint, which must be also topologically robust. Secondly, the proof of \cite{ELP} is given via a perturbative argument and it is crucial that the Beltrami fields provide explicit global smooth solutions of the Navier-Stokes equations. However, Beltrami fields have infinite energy on $\R^3$ and we wish to 
avoid infinite energy solutions. Any attempt to localize them will force us to work with reference solutions that are 
not explicit anymore. To overcome these problems our strategy is as follows:
\begin{enumerate}
[$i)$]
\item To state that two vector fields are not topologically equivalent we will count the number of hyperbolic zeros. This property is topologically robust.
\item Given a Beltrami field $B$ of frequency $N$ (namely $\curl B = N B$) one can consider a bump function $\psi$ such that $\|\psi B\|_{H^r}<\infty$. Then, to ensure the divergence-free condition we can consider the vector field $\curl(\psi B)$, which will be a perturbation of $N \psi B$ for an appropriate choice of the function $\psi $.
\end{enumerate}

We recall that ``counting the zeros of a vector field" has been used as a topological constraint by the authors and Pedro Caro in \cite{CCL} (see also \cite{Hyp22}) to construct examples of smooth periodic solutions of the two-dimensional MHD system showing {\em magnetic reconnection}. The latter represents, at least from a purely analytical point of view, the analogue of the vortex reconnection in the contest of Magnetohydrodynamics. Concerning topological results for MHD equations, we also mention the recent work \cite{EP-top} where a sophisticated counting argument has been used to prove obstructions to topological relaxation.
Our second main result is the following.

\begin{mainthmdue*}\label{thm:main}
Given any constants $\nu$, $T>0$, there exists a (small) smooth divergence-free vector field $u_0:\R^3\to\R^3$ such that \eqref{eq:ns} admits a unique global smooth solution $u$ with initial datum $u_0$, such that the vortex lines at time $t=0$ and $t=T$ are not topologically equivalent, meaning that there is no homeomorphism of $\R^3$ into itself mapping the vortex lines of $u(0,\cdot)$ into that of $u(T,\cdot)$.
\end{mainthmdue*}

More visually, what will happen is that the vorticity will have no zeros at the initial time $t=0$ while will have at least a (hyperbolic) zero at a later time $t=T$. This shows the change of the topology of the vortex lines. The scenario of vortex reconnection is structurally stable, as explained in Remark \eqref{RemReconnecitonTime} below. We also refer to this remark, and thus to Theorem \ref{thm:tempo}, for a quantification of the reconnection time.

We give some comments on the proof of the {\bf Theorem 2}. We will construct {\em small} global smooth solutions of the Navier-Stokes equations. The initial datum will be of the form
$$
u_0:= 
\rho \left( u_0^1 +  u_0^2 \right),
$$
where
$$
u_0^1:= \curl(\phi B_{N}), \qquad 
u_0^2 :=  e^{-\nu T\Delta} \curl(\psi W),
$$
and
\begin{enumerate}
[$i)$]
\item $B_{N}$ is a Beltrami field (in fact an ABC flow) which has no zeros, while $W$ is an eigenvector of the operator $\curl\curl$ which has some hyperbolic zeros (in fact infinitely many);
%\item Both $B_{N}$ and $W$ are .
\item $\phi$ and $\psi$ are two positive fast-decaying functions carefully chosen;
\item $N$ will be chosen very large: this will guarantee that at time 
$t=0$ the vector field~$ \omega_0 := \curl u_0$ can be considered as a perturbation of $\omega_0^1 := \curl u_0^1$, while at time $t=T$ the solution $\omega(T,\cdot)$ is arbitrarily close to a suitable rescaled version of $\curl\curl\psi W$;
\item $\rho$ will be chosen small enough, depending only on the frequency $N$ of $B_{N}$, to guarantee the global-in-time existence of our solutions.
%namely 
%$\delta e^{-T\Delta} \curl(\psi W)$ can be consider as a perturbation of $\curl(\phi B_{N})$.
\end{enumerate}

\begin{rem}\label{RemReconnecitonTime}
The reconnection scenario we present is not instantaneous, as we will show in Theorem \ref{thm:tempo}: this means that our solution do not have zeros not only for $t=0$, but also for any $0<t\ll 1$ small enough. This rules out the occurrence of zeros ``emerging" instantaneously from infinity. Furthermore, it will be clear by the proof that our construction is stable with respect to perturbations of the initial datum and/or the target time. We will also show that the reconnection happens on a viscous time scale. The latter means that the time of the reconnection is of order $\mathcal{O}(\nu^{-1})$: more precisely, we will show in Theorem \ref{thm:tempo} that the reconnection must happens at a time of order $\mathcal{O}(\nu^{-1}N^{-2})$ where $N$ is the frequency of the Beltrami field $B_N$ in the definition of $u_0$.    
\end{rem}

Finally, we point out that {\bf Theorem 2} provides also a bifurcation result for the Navier-Stokes equations on $\R^3$. In the two-dimensional {\it periodic} case, a similar scenario has been proposed in \cite{LS1, LS2}. In those papers, the authors study topology change of streamlines under short time evolution. The method of proof is completely different from ours and relies on a Taylor expansion in time, which allows them to compute the $\mathcal{O}(t)$ correction on the chosen data to show that the stream function of the solution has a number of critical points that changes in time. Exhibiting bifurcations in the periodic setting is easier, besides the aforementioned papers we also refers the reader to \cite{CCL, Hyp22, ELP}.

\subsection*{Organization of the paper} The paper is divided as follows. In Section 2 we set the notations and we recall some harmonic analysis tools together with some results on the heat kernel and on the Navier-Stokes equations. In Section 3 we provide the construction of the large initial data and then we prove the relative well-posedness result, this leads us to {\bf Theorem 1}. Then, in Section 4 we build the initial data for the reconnection result and we provide the proof of {\bf Theorem 2}. Lastly, in Section 5 we discuss the stability with respect to time of the reconnection scenario of {\bf Theorem 2}.

\section{Notations and preliminaries}\label{Sec:Notations} 
Throughout the paper, $C$ will denote a positive constant whose value can change line by line. We will use the symbol $a \lesssim b$ whenever there exists an absolute constant $C>0$ such that $a\leq Cb$.
\subsection{Functional spaces}
For any given positive integer $m$ and any given $d$-dimensional vector field $w:\R^d\to\R^d$, we define the quantity
$$
|\nabla^m w|^2:= \sum_{|\alpha|=m} |\partial^\alpha w|^2,
$$
where $\alpha \in \N^d$ is a multi-index, and then the Sobolev norm of $w$ can be defined as
$$
\|w\|^2_{W^{r, p}}:=\sum_{m=0}^r\int_{\R^d}|\nabla^m w(x)|^p \de x,
$$
where $r$ is a given positive integer. We denote as usual $W^{r, 2} = H^{r}$. Moreover, we shall denote with $\dot{H}^r$ the classical homogeneous Sobolev spaces, i.e. the space
$$
\dot{H}^r(\R^d):=\left\{ f\in \mathcal{S}': \hat{f}\in L^1_\mathrm{loc}(\R^d),\quad\mbox{and }\sum_{|\alpha|=m} |\partial^\alpha f|^2<\infty\right\},
$$
where with $\mathcal{S}'$ we denote the space of tempered distributions, and $\hat{f}$ denotes the Fourier transform of $f$.\\

We now recall the classical Littlewood-Paley decomposition, we use the notations of \cite{ChG09}.
Let $\varphi\in \mathcal{S}(\R^3)$ be a Schwartz function that the Fourier transform $\widehat{\varphi}$ satisfies 
\begin{equation}
\widehat{\varphi}(\xi)=\begin{cases}
1\qquad \mbox{for }|\xi|\leq 1,\\
0 \qquad \mbox{for }|\xi|>2.
\end{cases}
\end{equation}
For $j\in\Z$ we define the function $\varphi_j(x):=2^{3j}\varphi(2^j x)$, and the Littlewood-Paley operators 
\begin{equation}
S_j:=\varphi_j*\cdot,\qquad \Delta_j:=S_{j+1}-S_j.
\end{equation}
Note that the Fourier transform of $\Delta_j$ is supported on the dyadic annulus $\{\xi\in\R^3:2^{j-2}<|\xi|<2^{j}\}$.\\
\begin{defn}
Let $f\in\mathcal{S}'(\R^3)$. Then $f$ belongs to the homogeneous Besov space $\dot{B}^s_{p,q}(\R^3)$ if and only if
\begin{itemize}
\item the partial $\sum_{j=-m}^{m} \Delta_j f$ converges towards $f$ as a tempered distribution;
\item the sequence $\e_j=2^{js}\|\Delta_j f\|_{L^p}$ belongs to $\ell^q(\Z)$.
\end{itemize}
\end{defn}
With this definition, the Besov norm $\|\cdot\|_{\dot{B}^{s}_{p, q}}$ is defined  for all $1\leq p,q\leq \infty$ and $s\in\R$ as
\begin{equation}\label{Def:BesovNorm}
\| f \|_{\dot{B}^{s}_{p, q}} := \left( \sum_{j \in \mathbb{Z}} 2^{j q s} \| \Delta_j f\|_{L^p}^q \right)^{1/q}.
\end{equation}
Moreover, if $s<0$, one also have the equivalent norm
\begin{equation}
\| f \|_{\dot{B}^{s}_{p, q}} \sim \|t^{-s/2}\|e^{t\Delta}f\|_{L^p}\|_{L^q\left(\R^+;\frac{\de t}{t}\right)},
\end{equation}
where $e^{t\Delta}$ denotes the heat kernel.

\subsection{Harmonic analysis tools}\label{Sec:Decoupling}
Let $|D|^k$ be the Fourier multiplier with symbol $|\xi|^k$. Define the operator $P_{\leq \rho}$ as a smooth Fourier projection on the ball $B_{\rho}(0)$ of radius $\rho$ and center zero, namely 
\begin{equation}
\widehat{P_{ \leq \rho} f} (\xi):= \chi\left( \frac{\xi}{\rho} \right) \widehat{f}(\xi),
\end{equation}
with $\chi \geq 0$ being a smooth cut-off of the ball $B_{1/2}(0)$. More precisely 
$$
\chi(\xi) = \begin{cases}
1 \qquad &\mbox{for } \xi\in B_{1/2}(0),\\
0 &\mbox{for } \xi\in\R^3\setminus\bar{B}_{1}(0).
\end{cases}
$$
The complementary projection is defined as  $P_{> \rho} := 1 - P_{\leq \rho}$.
We now recall some harmonic bound as the Bernstein inequalities, that will be used in the 
following forms (see \cite{BCD, Lema-Rieusset}).
\begin{prop} Let $\rho>0$ and $1 \leq p \leq q \leq  \infty$. Then, for any $f\in L^p(\R^3)$ and $\rho >0$ 
we have that
\begin{align}\label{BernsteinIneq1}
\| P_{\leq \rho} f \|_{L^q} &\lesssim 
\rho^{3\left( \frac1p - \frac1q \right)} \| P_{\leq \rho}  f \|_{L^p},\\%, \qquad \rho >0, \quad 1 \leq p \leq q \leq  \infty,
\label{BernsteinIneq3}
\| \nabla P_{\leq \rho} f \|_{L^p} &\lesssim 
\rho \| P_{\leq \rho}  f \|_{L^p},\\%, \qquad \rho >0, \quad 1 \leq p \leq q \leq  \infty,
\label{BernsteinIneq2}
\| P_{> \rho} f \|_{L^p} &\lesssim 
\rho^{-k} \| P_{> \rho} |D|^k f \|_{L^p}.%, \qquad \rho >0, \quad p \in [1, \infty],
\end{align}
\end{prop}

We now recall some classical estimates on the heat kernel that we will often use in the following, see \cite{Lema-Rieusset, RRS}.
\begin{lem}\label{lem:heat kernel classica}
Let $e^{t\Delta}$ be the $3$-dimensional heat kernel. Then, for all $1\leq q\leq p\leq \infty$ and for any $t>0$ we have that
\begin{align}\label{HeatFlowBound}
\|e^{t\Delta}\, f\|_{L^p}&\leq \frac{C}{t^{\frac{3}{2}\left(\frac{1}{q}-\frac{1}{p}\right)}}\|f\|_{L^q},\\
\|e^{t\Delta}\,\nabla f\|_{L^p}&\leq \frac{C}{t^{\frac{1}{2}+\frac{3}{2}\left(\frac{1}{q}-\frac{1}{p}\right)}}\|f\|_{L^q}.
\end{align}
Moreover, for any $r\in\R$ the following bound holds
\begin{equation}
\|e^{t\Delta}\,\nabla f\|_{H^r}\leq \frac{C}{\sqrt{t}} \|f\|_{H^r}.\label{HeatFlowBound3}
\end{equation}
\end{lem}

Lastly, we recall a lemma that describes the action of the heat kernel on distributions with Fourier transforms supported in an annulus, see Lemma 2.4 in \cite{BCD}.
\begin{lem}\label{lemma chemin}
Let $\mathcal{C}\subset \R^3$ be an annulus. There exists two positive constants $c,C$ such that for any $1\leq p\leq \infty$ and any couple $(t,\lambda)$ of positive real numbers, we have
\begin{equation}
\mbox{supp }\hat{u}\subset\lambda\mathcal{C}\Longrightarrow \|e^{t\Delta}u\|_{L^p}\leq Ce^{-ct\lambda^2}\|u\|_{L^p}.
\end{equation}
\end{lem}
\begin{rem}
We point out that the value of $c$ and $C$ can be determined from the inner and outer radius of the annulus $\mathcal{C}$, respectively. This will be used repeatedly in the proofs below.
\end{rem}

\subsection{Zeros of a vector field}
A point $x_0\in\R^3$ is a zero of a vector field $v\in C^0(\R^3)$ if $v(x_0)=0$. A zero $x_0$ of a vector field $v\in C^1(\R^3)$ is said to be {\em non-degenerate} if $\nabla v(x_0)$ is an invertible matrix. A non-degenerate zero is said to be {\em hyperbolic} if $\nabla v(x_0)$ has no eigenvalues with zero real part. We recall that a vector field is structurally stable in a neighborhood of a hyperbolic zero. It is easy to show, using the implicit function theorem,
that hyperbolic zeros are preserved by $C^1$ perturbations. 

\subsection{Beltrami fields}\label{Sec:BeltramiTaylor}

We now introduce the main mathematical objects that we need in our constructions. These are the so-called Beltrami fields.

A vector field $B : \R^3\to\R^3$ is called a {\em Beltrami field} with frequency $\lambda$ if it is an eigenfunction of the $\curl$ operator with eigenvalue $\lambda \in \R$, i.e.
\begin{equation}\label{eq:LambdaBeltrami}
\curl B =\lambda B.
\end{equation}
We will restrict our attention to Beltrami fields of non-zero frequency ($\lambda \neq 0$). Note that \eqref{eq:LambdaBeltrami} automatically implies that $\dive B=0$.
Moreover, it is easy to check that $B$ satisfies additionally
\begin{equation}\label{BeltrProp}
(B\cdot\nabla)B=\nabla\left(\frac{|B|^2}{2}\right),\,\,\qquad
\Delta B=-\lambda^2B.
\end{equation}

From \eqref{BeltrProp} it follows that the divergence free field $e^{-t\lambda^2} B$ is a solution of the 
Navier-Stokes equation with pressure $\frac12 e^{-2 t\lambda^2} |B|^2$.

\subsection{Some results on the Navier-Stokes equations}
We recall a classical result on the existence of global strong solutions of the Navier-Stokes equations, see 
Theorem 5.6 of \cite{BCD} and Theorem 7.1 in \cite{RRS}.
\begin{thm}\label{lem:stima-ns}
Let $\nu>0$ and $u_0 \in H^{1/2}(\R^3)$ be a divergence-free vector field. There exists a 
(sufficiently small) constant 
$C$ such that if $\|u_0\|_{H^{1/2}}  \leq C \nu$, then there exists a unique solution $u\in C(\R^+; H^\frac12(\R^3))\cap L^2(\R^+;\dot{H}^{\frac32}(\R^3))$ such that
\begin{equation}
\|u(t,\cdot)\|_{H^\frac12} \leq 2C\nu.
\end{equation}
Moreover, if the initial datum $u_0\in  H^r(\R^3)$ for some $r>\frac12$ then
$$
u\in C(\R^+; H^r(\R^3))\cap L^2(\R^+;\dot{H}^{r+1}(\R^3)).
$$
\end{thm}
For our purposes we will need a quantitative estimate on the $H^r$-norm of the solution provided by Theorem \ref{lem:stima-ns} of the type
\begin{equation}\label{eq:stima Hs quantitativa}
\|u(t,\cdot)\|_{H^r} \leq C_r \|u_0\|_{H^r}.
\end{equation}
We will prove this estimate under the assumptions of {\bf Theorem 2} in Lemma \ref{stima norme ns} in the Appendix. 

We conclude this section by recalling the main result of \cite{ChG09}. Define the space-time norm
\begin{equation}\label{CGSmallCond}
\| F \|_{E} := \int_{0}^{\infty} \| F(t,\cdot) \|_{\dot{B}^{-1}_{\infty, 1}} dt + 
\sum_{j \in \mathbb{Z}} 2^{-j} \left( \int_{0}^{\infty} \| \Delta_j F(t,\cdot) \|^2_{L^{\infty}} t \, dt \right)^{1/2},
\end{equation}
and the quantity
$$
\Omega_{u_0} := 
%(w \cdot \nabla ) w  - \frac{\lambda^2}{2} e^{-2t \lambda^2} \nabla (\phi^2 |B_\lambda|^2 )  =
(e^{t\Delta} u_0 \cdot \nabla ) e^{t\Delta} u_0  .
$$

\begin{thm}\cite{ChG09}\label{teo:Chemin-Gallagher}
There is a constant $C^*>0$ such that the following result holds.
Let $u_0\in \dot{H}^{\frac12}(\R^3)$ be a divergence free vector field. Suppose that
\begin{equation}\label{SmallnessConditionCG}
\| \mathbb{P} \, \Omega_{u_0} \|_{E} \leq \frac{1}{C^*} \exp\left(-C^*  \| u_0\|_{\dot{B}^{-1}_{\infty, 2}}^4  \right).
\end{equation}
Then, there is  a unique, global solution $u$ to the Navier--Stokes equations 
associated to $u_0$, satisfying
$$
u\in C_b(\R^+; \dot{H}^{\frac12}(\R^3))\cap L^2(\R^+; \dot{H}^{\frac32}(\R^3)).
$$
\end{thm}

\section{Solutions with large data} 
The aim of this section is to prove {\bf Theorem 1}. First of all we provide the construction and the main properties of the class of initial data that we will consider in {\bf Theorem 1}. Such initial data  comes from a suitable localization of a Beltrami field $B_{\lambda}$ with frequency 
$\lambda$ and with $\| B_{\lambda} \|_{L^{\infty}} < \infty$.  By rescaling, there is no loss of generality to restrict to  $\| B_{\lambda} \|_{L^{\infty}} = 1$ and to encode the $L^\infty$ size of the initial datum in a constant $M>0$. Recall that if $B$ is Beltrami so it is $\rho B$, for any $\rho\in\R$. For simplicity, in this section we are going to set $\nu=1$.

\subsection{Construction of the initial data}
As already mentioned in the introduction, our initial data, coming from the localization of $B_\lambda$, have the following form
\begin{equation}\label{LocBeltramiAgain}
u_0 := M  \curl(\phi_{L} B_{\lambda}) , \qquad M, L > 0,
\end{equation}
where the localizing function $\phi_{L}$ is defined as
$$
\phi_L = \mathcal{D}_L \phi, \qquad 
\mathcal{D}_L : f  \to f\left( \frac{\cdot}{L} \right),  
$$
for some $\phi \in W^{s,1} \cap W^{s, \infty}(\R^3)$; we will choose the regularity $s$ to be very large. The goal of this subsection is to prove Lemma \ref{ForcingTermLemma}: it will be the main tool to show that our initial data satisfies the nonlinear smallness condition~\eqref{NonlSmallIntro} and thus prove {\bf Theorem 1}.\\

In Lemma~\ref{ForcingTermLemma} we establish an $L^p$ estimate for the quantity 
\begin{equation}\label{fdnsjkdjngjkdsljdfgn}
\mathbb{P} \, (e^{t\Delta} u_0 \cdot \nabla ) e^{t\Delta} u_0,  
\qquad u_0 := M  \curl(\phi_{L} B_{\lambda}) ,
\end{equation}
with a suitable time decay (we recall that $\mathbb{P}$ denotes the classical Leray projector). To do so, using the 
algebraic properties of Beltrami fields we will rewrite
$$
(e^{t\Delta} u_0 \cdot \nabla ) e^{t\Delta} u_0 \simeq \nabla A + \text{Remainders} ,
$$
in such a way that we will reduce the estimate of the $L^p$ norm of $\mathbb{P} \, (e^{t\Delta} u_0 \cdot \nabla ) e^{t\Delta} u_0$ to that of a sum 
of $L^p$ norms of remainders terms. We will identify terms as remainders if 
\begin{itemize}
\item at least one partial derivative falls on $\phi_L$;
\item they have a suitable commutator structure, associated to the operator 
$[e^{t\Delta} , \phi_L]$. %the one in Lemma \ref{};
\end{itemize}
In the first case we will gain a factor $L^{-1}$ by $\partial_j \phi_L$, in the second case we will gain this factor from the commutator using Lemma \ref{Lemma:dec}. \\

We are now ready to prove the lemmas which will allow us to handle the remainder terms.

\begin{lem}\label{LemmaT1}
Let $B_{\lambda} \in L^{q} \cap L^{\infty}(\R^3)$ be a Beltrami field with frequency $\lambda$, $\| B_{\lambda}\|_{L^{\infty}} =1$ and $\| B_{\lambda}\|_{L^{q}} = R$, and let $\phi \in W^{s,1} \cap W^{s, \infty}(\R^3)$.
Then, for any $1 \leq p \leq q \leq \infty$ and for all $0 < \varepsilon \ll 1 $ one can take sufficiently large values of $s$ and of the dilation factor~$L$ such that %, for all $q \geq p$ we have
\begin{equation}\label{CommBound2}
\| e^{t\Delta} ( \nabla\phi_L\wedge B_{\lambda} )  \|_{L^p} \lesssim 
\left\{ \begin{array}{ll}
t^{-\frac32\left( 1 - \frac1p \right)} L^{- \theta s} & \mbox{for $t > L^{\varepsilon}$}, \\
R L^{-1 + \frac{3}{p} - \frac{3}{q}} & \mbox{for $t \leq L^{\varepsilon}$} ,
\end{array}
\right. 
\end{equation}\
where $\theta \in (0,1)$ is an absolute constant.
\end{lem}

\begin{proof}
Throughout the proof we will choose $s$ sufficiently large so that all the Sobolev norms of $\phi$ 
that we need are finite. We have

$$
 \| e^{t\Delta} ( \nabla\phi_L\wedge B_{\lambda} )  \|_{L^{p}} =
 L^{-1} \| e^{t\Delta} ( (\mathcal{D}_L  \nabla \phi) \wedge B_{\lambda} )  \|_{L^{p}},
$$
where we used $$\nabla \phi_L  = \nabla \mathcal{D}_L \phi = L^{-1} \mathcal{D}_L  \nabla \phi.$$ 

We first focus on short times $t \leq L^{\varepsilon}$.  
%(fix $0 < \varepsilon \ll 1$).
We use the properties of the heat kernel and H\"older inequality to obtain the estimate %use the inequality %have 
$$
L^{-1} \| e^{t\Delta} ( (  \mathcal{D}_L    \nabla \phi) \wedge B_{\lambda} ) \|_{L^{p}}
\leq  L^{-1} \|  \mathcal{D}_L   \nabla \phi \|_{L^{\frac{pq}{q-p}}} \| B_{\lambda}  \|_{L^{q}} 
\lesssim R L^{-1 + \frac{3}{p} - \frac{3}{q}} \|  \nabla \phi \|_{L^{\frac{pq}{q-p}}} %\| B_{\lambda}  \|_{L^{\infty}} 
\lesssim R L^{-1 + \frac{3}{p} - \frac{3}{q}}.
$$
To handle large times $t > L^{\varepsilon}$ we will distinguish the small and large frequencies of the localizing function~$\mathcal{D}_L \nabla \phi$.
Thus we (smoothly) project $\mathcal{D}_L \nabla \phi$ as follows%, on small and large frequencies
\begin{equation}\label{ProjectionLF}
\mathcal{D}_L\nabla \phi = P_{\leq \lambda/2} (\mathcal{D}_L  \nabla \phi)  +
P_{> \lambda/2} (\mathcal{D}_L  \nabla \phi), 
\end{equation}
where $\widehat{P_{ \leq \rho} f} := (D_{\rho} \chi) \widehat{f}$ with $\chi \geq 0$ being a smooth cut-off of the ball 
$B_{1/2}(0)$ (see Section \ref{Sec:Decoupling}). The complementary 
projection is defined as  $P_{> \rho} := 1 - P_{\leq \rho}$.

On large frequencies of $\mathcal{D}_L \nabla \phi$ (and large times) we estimate
\begin{align}
 L^{-1} \| e^{t\Delta} ( ( P_{> \lambda/2} \mathcal{D}_L   \nabla \phi) \wedge B_{\lambda} ) \|_{L^{p}} 
& \leq 
  t^{- \frac{3}{2} \left( 1- \frac1p \right)} L^{-1} \| P_{> \lambda/2} \mathcal{D}_L   \nabla \phi \|_{L^1} \| B_{\lambda}  \|_{L^{\infty}} \nonumber
\\ \nonumber
& 
 \lesssim t^{- \frac{3}{2} \left( 1- \frac1p \right)} L^{-1} \lambda^{-k} 
  \| P_{> \lambda/2}|D|^k \mathcal{D}_L   \nabla \phi \|_{L^1} \| B_{\lambda}  \|_{L^{\infty}}
\\ \nonumber
& 
\lesssim  t^{- \frac{3}{2} \left( 1- \frac1p \right)} L^{-1 -k} \lambda^{-k} 
  \|  P_{> \lambda/2} \mathcal{D}_L  |D|^k \nabla \phi \|_{L^1} \| B_{\lambda}  \|_{L^{\infty}}
\\ \nonumber
& 
\lesssim
  t^{- \frac{3}{2} \left( 1- \frac1p \right)} L^{2 - k} \lambda^{-k} 
  \|   |D|^k \nabla \phi \|_{L^1} \| B_{\lambda}  \|_{L^{\infty}}
 \\ %\nonumber
& \lesssim   t^{- \frac{3}{2} \left( 1- \frac1p \right)} L^{2-k} \lambda^{-k}  ,
%old result (essentially the same) \qquad L^{\varepsilon + 2 - \frac{k}{2}} \lambda^{\frac32 -\frac{k}{2}},
\end{align}
where we used, in order, Lemma \ref{lem:heat kernel classica}, the Bernstein inequality \eqref{BernsteinIneq2} 
and the properties of the dilation operator. This is compatible with \eqref{CommBound2}
 choosing $s=Ck$ with $\theta<1/C$ and taking $L \gg \lambda + \lambda^{-1}$.

For small frequencies of $\mathcal{D}_L \nabla \phi$ and large times we use the fact that 
$$
\text{supp\,}\mathcal{F} \left( (P_{\leq \lambda/2 } \mathcal{D}_L    \nabla \phi ) \wedge B_{\lambda} \right)
\subset \left\{ \frac{\lambda}{2} \leq |\xi| \leq \frac32 \lambda \right\},
$$
where $\mathcal{F}$ denotes the Fourier transform. By this support condition and Lemma \ref{lemma chemin} we have 
\begin{align}
  L^{-1}  \| e^{t\Delta}  ( ( P_{\leq \lambda/2} \mathcal{D}_L   \nabla \phi)  \wedge B_{\lambda} ) \|_{L^{p}} 
 &\lesssim    L^{-1} e^{- \frac{t}{4} \lambda^2} 
\|  ( P_{\leq \lambda/2} \mathcal{D}_L   \nabla \phi) \wedge B_{\lambda}  \|_{L^{p}}  \nonumber
\\ \nonumber
&
\lesssim  L^{-1}  e^{- \frac{L^{\varepsilon}}{4} \lambda^2} e^{- \frac{t}{8} \lambda^2} 
\|   \mathcal{D}_L  \nabla \phi \|_{L^{p}} \| B_{\lambda}  \|_{L^{\infty}}  
\\ \nonumber
&
\lesssim   L^{-1 + \frac{3}{p}}   e^{- \frac{L^{\varepsilon}}{4} \lambda^2} e^{- \frac{t}{8} \lambda^2}
\|   \nabla \phi \|_{L^{p}} \| B_{\lambda}  \|_{L^{\infty}}  
\\
&\lesssim  L^{-1}  e^{- \frac{L^{\varepsilon}}{4} \lambda^2} e^{- \frac{t}{8} \lambda^2}.
\end{align}
These estimates lead to the desired statement taking $L \gg \lambda + \lambda^{-1} $. 
\end{proof}

The following lemma provides the bound for the commutator expression.

\begin{lem}\label{Lemma:dec}  
Let $B_{\lambda} \in L^q \cap L^{\infty}(\R^3)$ be a Beltrami field with frequency $\lambda$, $\| B_{\lambda}\|_{L^{\infty}} =1$
and $\| B_{\lambda}\|_{L^{q}} = R$, and let $\phi \in W^{s,1} \cap W^{s, \infty}(\R^3)$. 
%Let $p \in [1, \infty]$ and  a localizing function of order $s$ as in Definition \ref{Def:sLocalizing}. 
Then, for any $1 \leq p \leq q \leq \infty$  and for all $0 < \varepsilon \ll 1 $ one can take sufficiently large values of  $s$ and of the dilation factor~$L$ such that 
\begin{equation}\label{CommBound} 
\| [e^{t\Delta} , \phi_L]  B_{\lambda}  \|_{L^p} \lesssim 
\left\{ \begin{array}{ll}
t^{-\frac32\left( 1 - \frac1p \right)} L^{- \theta s}  & \mbox{for $t > L^{\varepsilon}$} \\
  R L^{-1 + \frac{\varepsilon}{2} + \frac{3}{p} - \frac{3}{q}} & \mbox{for $t \leq L^{\varepsilon}$} ,
\end{array}
\right. 
\end{equation}\
where $\theta \in (0,1)$ is an absolute constant.  %and $C >1$ are absolute constants.
\end{lem}

\begin{rem}
As will be clear by the proof, the short time estimate can be strengthened to
$$
\| [e^{t\Delta} , \phi_L]  B_{\lambda}  \|_{L^p}\lesssim \sqrt{\nu t} R L^{-1 + \frac{3}{p} - \frac{3}{q}},
$$
however we will not need this improvement.
\end{rem}

\begin{proof}
Throughout the proof we will choose $s$ sufficiently large so that all the Sobolev norms of $\phi$ 
that we need are finite. 
Using $e^{t \Delta} B_\lambda = e^{- t \lambda^2} B_\lambda$, we can compute explicitly
$$
[e^{t\Delta} , \phi_L]  B_{\lambda} = (e^{t\Delta} - e^{- t \lambda^2}) (\phi_L B_\lambda).
$$

We first prove the statement for $t > L^{\varepsilon}$. 
We (smoothly) project $\phi_L$ on small and large frequencies, as in \eqref{ProjectionLF}, and we first focus on 
large frequencies for $\phi_L$ (and large times). Here we can bound the terms in the difference separately. Using Lemma \ref{lem:heat kernel classica} and the Bernstein inequality \eqref{BernsteinIneq2} 
we get
\begin{align}
\| e^{t\Delta} ( (P_{> \lambda/2} \phi_L) B_\lambda) \|_{L^{p}} 
& \lesssim t^{-\frac32\left( 1 - \frac1p \right)}
 \| P_{> \lambda/2} \mathcal{D}_L  \phi \|_{L^1} \| B_{\lambda}  \|_{L^{\infty}} \nonumber
\\ \nonumber
&
\lesssim t^{-\frac32\left( 1 - \frac1p \right)} \lambda^{-k}
 \| P_{> \lambda/2}|D|^k \mathcal{D}_L  \phi \|_{L^1} 
 \\ \nonumber
&
\lesssim t^{-\frac32\left( 1 - \frac1p \right)} \lambda^{-k} L^{-k}
 \|  P_{> \lambda/2} \mathcal{D}_L |D|^k \phi \|_{L^1} 
 \\ %\nonumber
&
\lesssim t^{-\frac32\left( 1 - \frac1p \right)} \lambda^{-k} L^{-k +3}
 \|   |D|^k \phi \|_{L^1} 
 \lesssim t^{-\frac32\left( 1 - \frac1p \right)} \lambda^{-k} L^{-k +3}.
\end{align}
Moreover, we also immediately have
\begin{align}\label{VeryImm}
\| e^{- t \lambda^2} (  (P_{> \lambda/2} \phi_L) B_\lambda) \|_{L^{p}} \nonumber
& \lesssim 
e^{-  t \lambda^2} \|   \phi_L \|_{L^{p}} \| B_\lambda \|_{L^{\infty}} 
\\ %\nonumber
& \lesssim 
e^{-  \frac12 L^{\varepsilon} \lambda^2 }
e^{-  \frac12 t \lambda^2 } L^{\frac3p} \|   \phi \|_{L^{p}} % \| B_\lambda \|_{L^{\infty}} 
\lesssim e^{-  \frac12 L^{\varepsilon} \lambda^2 }
e^{-  \frac12 t \lambda^2 } L^{\frac3p} , \qquad t > L^{\varepsilon}.
\end{align}
%Note that on the right hand side of the estimate above we only have the dominant contribution from
%$e^{t\Delta} (\phi_L B_\lambda)$, since in the regime $t > L^{\varepsilon}$,  $L \gg \lambda$ we clearly have $e^{- t \lambda^2} \lesssim t^{-3/2}$.
We now look at small frequencies of $\phi_L$ and large times: we use the support condition 
$$
\text{supp\,}\mathcal{F} \left( (P_{\leq \lambda/2 }  \phi_L ) B_{\lambda} \right)
\subset \left\{ \frac{\lambda}{2} \leq |\xi| \leq \frac32 \lambda \right\},
$$
together with Lemma \ref{lemma chemin} to get 
\begin{align}
 \| e^{t\Delta} ( ( P_{\leq \lambda/2}  \phi_L)  B_{\lambda} ) \|_{L^{p}} 
& 
\lesssim    e^{- \frac{t}{4} \lambda^2}
\|  \mathcal{D}_L    \phi \|_{L^{p}}   \| B_{\lambda}  \|_{L^{\infty}}  
\\ \nonumber
&
\lesssim  L^{\frac{3}{p}}  e^{- \frac{L^{\varepsilon}}{8} \lambda^2} e^{- \frac{t}{8} \lambda^2} 
\|      \phi \|_{L^{p}}   \| B_{\lambda}  \|_{L^{\infty}}  
\\ \nonumber
&
\lesssim  L^{\frac{3}{p}}  e^{- \frac{L^{\varepsilon}}{8} \lambda^2} e^{- \frac{t}{8} \lambda^2}  , \qquad t > L^{\varepsilon}. 
\end{align}
For the term $\| e^{- t \lambda^2} (  (P_{\leq \lambda/2} \phi_L) B_\lambda) \|_{L^p}$ we can argue as in \eqref{VeryImm} 
(the analysis is in fact even easier, we do not even need to distinguish between high and low frequencies).
These inequalities conclude the proof for times $t > L^{\varepsilon}$
as long as we choose $L \gg \lambda + \lambda^{-1}$.

We now focus on small times, i.e. $t \leq L^{\varepsilon}$. Here we must use a finer choice of the projection threshold, namely
\begin{equation}%\label{ProjectionLF}
\nabla \phi_L = P_{\leq \frac{ L^{\varepsilon} }{L} }\, \phi_L  +
P_{> \frac{L^{\varepsilon} }{L} }\, \phi_L.
\end{equation}

For large frequencies of $\phi_L$ we have, using Lemma \ref{lem:heat kernel classica} and Bernstein inequality \eqref{BernsteinIneq2} %(compare with \eqref{})
\begin{align}\label{BernstStrongg}
\| e^{t\Delta} (( P_{> \frac{L^{\varepsilon} }{L}} \phi_L) & B_\lambda) \|_{L^{p}} 
+ 
\| e^{- t \lambda^2} ( (P_{> \frac{L^{\varepsilon} }{L}} \phi_L) B_\lambda) \|_{L^{p}} \nonumber
\\ \nonumber
&\lesssim 
 \| P_{> \frac{L^{\varepsilon} }{L} } \mathcal{D}_L  \phi \|_{L^p} \| B_{\lambda}  \|_{L^{\infty}} 
\\ \nonumber
& \lesssim  L^{k(1 - \varepsilon)}  \| P_{> \frac{L^{\varepsilon} }{L} } |D|^k \mathcal{D}_L  \phi \|_{L^p} 
\\ \nonumber
&\lesssim  L^{-k \varepsilon} 
\|   P_{> \frac{L^{\varepsilon} }{L} } \mathcal{D}_L |D|^k \phi \|_{L^p}
\\ 
&
\lesssim 
L^{-k \varepsilon + \frac3p} 
\|  |D|^k \phi \|_{L^p} \lesssim L^{-k \varepsilon + \frac3p} .
\end{align}
Taking $k > \frac{3}{\varepsilon p}$ and $L \gg \lambda + \lambda^{-1}$ this also gives the desired inequality.

It remains to consider small times and small frequencies of $\phi_L$, that is the hardest case. Indeed, this is the only case in which the smallness will be provided by a cancellation in the commutator. Let $t\leq L^\e$ and denote by $G(t,x):=\frac{1}{(4\pi\nu t)^{3/2}}e^{-\frac{|x|^2}{4t}}$ the heat kernel. We write the commutator as
\begin{align*}
[e^{t\Delta} , \phi_L]  B_{\lambda}(x)&=e^{t\Delta}(\phi_L B_\lambda)-\phi_L e^{t\Delta}B_\lambda\\
&=\int_{\R^3}G(t,x-y)\phi_L(y)B_\lambda(y)\de y-\phi_L(x)\int_{\R^3}G(t,x-y)B_\lambda(y)\de y\\
&=\int_{\R^3}G(t,x-y)B_\lambda(y)(\phi_L(y)-\phi_L(x))\de y.
\end{align*}
Then, since $G(t,y)\de y$ is a probability measure and the function $f(s)=s^p$ is convex for $p\geq1$, we 
obtain from  
Jensen's inequality 
$$
|[e^{t\Delta} , \phi_L]  B_{\lambda}(x)|^p  \leq \int_{\R^3}G(t,x-y)|B_\lambda(y)|^p|\phi_L(y)-\phi_L(x)|^p\de y.
%\\
%&\leq \int_{\R^3}G(t,x-y)|B_\lambda(y)|^p|\phi_L(y)-\phi_L(x)|^p\de y.
$$
We now compute the $L^p$ norm: we integrate the above expression over $\R^3$ obtaining that
\begin{align*}
\|[e^{t\Delta} , \phi_L]  B_{\lambda}\|_{L^p}^p&\leq \int_{\R^3}\int_{\R^3}G(t,x-y)|B_\lambda(y)|^p|\phi_L(y)-\phi_L(x)|^p\de y\de x\\
&=\frac{1}{(4\pi\nu t)^{3/2}}\int_{\R^3}\int_{\R^3}e^{-\frac{|x-y|^2}{4\nu t}}|B_\lambda(y)|^p|\phi_L(y)-\phi_L(x)|^p\de y\de x\\
&=\frac{1}{\pi^{3/2}}\int_{\R^3}\int_{\R^3}e^{-|z|^2}|B_\lambda(y)|^p|\phi_L(y+\sqrt{4\nu t}z)-\phi_L(y)|^p\de z\de y,
\end{align*}
where in the last line we made the change of variables $x=y+\sqrt{4\nu t}z$. Then, we use the definition of $\phi_L$ and that
\begin{equation*}
|\phi_L(y+\sqrt{4\nu t}z)-\phi_L(y)|=L^{-1} \sqrt{4\nu t}  \left|\int_0^1
z \cdot \nabla\phi\left(\frac{y+s\sqrt{4\nu t}z}{L}\right)\de s\right|.
\end{equation*}
Using this and H\"older inequality on $[0,1]$ we arrive to 
\begin{align*}
\|[e^{t\Delta} , \phi_L]  B_{\lambda}\|_{L^p}^p&\lesssim L^{-p} (\nu t)^{\frac{p}{2}} \int_{\R^3}\int_{\R^3}\int_0^1 |z| e^{-|z|^2}|B_\lambda(y)|^p\left|\nabla\phi\left(\frac{y+s\sqrt{4\nu t}z}{L}\right)\right|^p \de s \de y\de z\\
&=L^{-p} (\nu t)^{\frac{p}{2}} \int_{\R^3}\int_0^1 |z| e^{-|z|^2}\underbrace{\left(\int_{\R^3}|B_\lambda(y)|^p\left|\nabla\phi\left(\frac{y+s\sqrt{4\nu t}z}{L}\right)\right|^p\de y\right)}_{(*)}\de s\de z.
\end{align*}
To bound $(*)$ we use H\"older's inequality with exponents $\frac{q}{p}$ and $\frac{q}{q-p}$: since the Lebesgue's norm is translation invariant, we obtain that
\begin{equation}
\int_{\R^3}|B_\lambda(y)|^p\left|\nabla\phi\left(\frac{y+s\sqrt{4\nu t}z}{L}\right)\right|^p\de y\leq L^{3(1-\frac{p}{q})}\|B_\lambda\|_{L^q}^p \|\nabla\phi\|_{L^\frac{pq}{q-p}}^p,
\end{equation}
which leads to the estimate
\begin{align*}
\|[e^{t\Delta} , \phi_L]  B_{\lambda}\|_{L^p}^p&\lesssim (\nu t)^{\frac{p}{2}} L^{3(1-\frac{p}{q})-p}\|B_\lambda\|_{L^q}^p\|\nabla\phi\|_{L^\frac{pq}{q-p}}^p\int_0^1\int_{\R^3} |z| e^{-|z|^2}\de x\de s,
\end{align*}
and then
\begin{equation}
\|[e^{t\Delta} , \phi_L]  B_{\lambda}\|_{L^p}\lesssim (\nu t)^{\frac{1}{2}} RL^{-1+\frac{3}{p}-\frac{3}{q}}.
\end{equation}
This concludes the proof.
\end{proof}

\begin{rem}\label{PartialDerivativesRemark}
The estimates from Lemma \ref{LemmaT1} and Lemma \ref{Lemma:dec} also hold for the partial derivatives of the left hand side, namely for the $L^p$ norms of $\partial_j [e^{t\Delta} , \phi_L]  B_{\lambda}$ and $\partial_j e^{t\Delta} ( \nabla\phi_L\wedge B_{\lambda} )$. In fact, when the partial derivative falls on $\phi_L$ we get an even better estimate (we gain $L^{-1}$). On the other hand, when the partial derivative falls on $B_{\lambda}$ we use that: (1) $\partial_j B_{\lambda}$ is a Beltrami field of frequency $\lambda$ and (2): 
\begin{equation}\label{NablaEstimateLambda}
\| \partial_j B_{\lambda} \|_{L^p} \lesssim \lambda 
\| B_{\lambda} \|_{L^p},
\end{equation}
for all $p \in [1, \infty]$, which is a consequence of Bernstein inequality \eqref{BernsteinIneq3} that we are allow to apply since $\mathrm{supp}\, \widehat{\partial_j B_\lambda} \subseteq \{ |\xi | = \lambda\}$. Combining (1) and (2) it is clear that we can use the same argument as in the proof of Lemmas \ref{LemmaT1}-\ref{Lemma:dec} to get the same estimate with an additional loss of $C \lambda$. In the short time estimates, this can be absorbed by an extra factor $L^{\varepsilon}$ that is harmless for our purposes, while in the long time estimate this simply means to choose $\theta$ slightly smaller). Similarly, the factor $R$ appearing in \eqref{CommBound}-\eqref{CommBound2} can be absorbed by an extra factor $L^{\varepsilon}$, choosing $L \gg R$. 
\end{rem}

The next lemma is the main result of this section and the main ingredient in the proof of {\bf Theorem 1}, that is the well-posedness for initial data that are localization of Beltrami fields in the sense of \eqref{LocBeltramiAgain}. We define for brevity
\begin{equation}\label{Def:OmegaW0}
\Omega_{u_0} := 
%(w \cdot \nabla ) w  - \frac{\lambda^2}{2} e^{-2t \lambda^2} \nabla (\phi^2 |B_\lambda|^2 )  =
(e^{t\Delta} u_0 \cdot \nabla ) e^{t\Delta} u_0.
% - \frac{\lambda^2}{2} e^{-2t \lambda^2} \nabla (\phi^2 |B_\lambda|^2 ) \,.
\end{equation}
and recall that 
$\mathbb{P}$ is the projection on the linear subspace of divergence free vector fields.

\begin{lem}\label{ForcingTermLemma}
Let $q \in (3, \infty)$ and $B_{\lambda} \in L^q \cap L^{\infty}(\R^3)$ be a Beltrami field with frequency $\lambda$ and with $\| B_{\lambda}\|_{L^{\infty}} =1$. Let $\phi \in W^{s,1} \cap W^{s, \infty}(\R^3)$. Then, for all $0 < \varepsilon \ll 1 $ one can take sufficiently large values of  $s$ and of the dilation factor~$L$ such that
\begin{equation}\label{CommBound2}
\| \mathbb{P} \, \Omega_{u_0}(t,\cdot) \|_{L^{r}} \lesssim 
\left\{ \begin{array}{ll}
t^{-3\left(1-\frac{1}{2r}\right)} L^{- \theta s} & \mbox{for $t > L^{\varepsilon}$}, \\
L^{-1 + \varepsilon+\frac3r - \frac3q} & \mbox{for $t \leq L^{\varepsilon}$} ,
\end{array}
\right.
\end{equation} 
for any $1\leq r<\infty$ satisfying 
\begin{equation}
\label{relazione r q}
\frac1r<\frac13+\frac1q-\frac\e3.
\end{equation}
\end{lem}
\begin{proof}
We set $R := \| B_{\lambda}\|_{L^{q}}$.
%Recall that $u_0  = \curl (\phi_L B_{\lambda} )$
By using standard vector identities and the definition of $B_{\lambda}$, it is immediate to check that
\begin{equation}\label{def:u_0^1Preq}
u_0 = \lambda \phi_L B_{\lambda} + \nabla\phi_L\wedge B_{\lambda}. % =: N\phi B_{N}+\Omega,
\end{equation}
Thus
we can write the free evolution of the initial datum as
\begin{align}\label{RecThis1}
e^{t\Delta}u_0 
& =  \lambda e^{ t\Delta} ( \phi_L B_{\lambda} ) + e^{t\Delta} ( \nabla\phi_L\wedge B_{\lambda} )
\nonumber\\ 
& = \lambda e^{- t \lambda^2} \phi_L    B_{\lambda}  +  \lambda [e^{ t\Delta} , \phi_L]  B_{\lambda}   
+ e^{ t\Delta} ( \nabla\phi_L\wedge B_{\lambda} ),
\end{align}
where we used $e^{ t\Delta} B_{\lambda} = e^{- t\lambda^2} B_{\lambda}$.
Then, we define
$$
R_1 = \lambda^2 e^{- 2 t \lambda^2}  \phi_L \left( (  B_{\lambda}   \cdot     \nabla )  \phi_L \right)     B_{\lambda} ,
\qquad
R_2 = \lambda [e^{t\Delta} , \phi_L]  B_{\lambda}   ,
\qquad
R_3 = e^{t\Delta} ( \nabla\phi_L\wedge B_{\lambda} ),
$$
and we
expand  
\begin{align}\label{MainExpansion}
 (e^{t\Delta} u_0 \cdot \nabla ) e^{t\Delta} u_0
 & = 
\lambda^2 e^{- 2t \lambda^2} \phi_L^2  (  B_{\lambda}   \cdot \nabla )   B_{\lambda} + R_1 + \lambda e^{-t \lambda^2} \sum_{j = 2,3}  ( \phi_L  ( B_{\lambda}  \cdot \nabla ) R_j
\nonumber\\
& + \lambda e^{-t \lambda^2}  \sum_{j = 2,3}  (R_j \cdot \nabla )  \phi_L    B_{\lambda}   +  \sum_{\substack{i,j =2,3}} (R_i \cdot \nabla) R_j.
\end{align}
The principle is that we identify as rest every term that contains a commutator or such that at least one 
partial derivative falls on $\phi_L$.
Then since 
$$
 (  B_{\lambda}   \cdot \nabla )   B_{\lambda} = \frac{1}{2} \nabla  |B_\lambda|^2,
$$ 
we notice that the dominant term on the r.h.s. of \eqref{MainExpansion} can be rewritten as
$$
\lambda^2 e^{-2 t \lambda^2} \phi_L^2  (  B_{\lambda}   \cdot \nabla )   B_{\lambda}
= \frac{\lambda^2}{2} e^{-2t \lambda^2} \nabla ( \phi^2 |B_\lambda|^2)   -  
 \frac{\lambda^2}{2}   e^{-2t \lambda^2} |B_{\lambda}|^2   \nabla \phi_L^2  = \nabla A  +  R_4 ,
 $$ 
that is the sum of a gradient and a 
new remainder 
 $$
 R_4 := - \frac{\lambda^2}{2}   e^{-2t \lambda^2} |B_{\lambda}|^2   \nabla (\phi_L^2).
 $$
Thus, since $\mathbb{P} \, \nabla A =0$, we arrive to 
\begin{equation}\label{errore qua}
\mathbb{P}\, \Omega_{u_0} = \mathbb{P}\left[
R_1 + R_4 +  \lambda e^{-t \lambda^2} \sum_{j = 2,3}  \phi_L  ( B_{\lambda}  \cdot \nabla ) R_j
 + \lambda e^{-t \lambda^2}  \sum_{j = 2,3}  
 (R_j \cdot \nabla )  \phi_L    B_{\lambda}   +  \sum_{\substack{i,j =2,3}} (R_i \cdot \nabla) R_j\right],
\end{equation}
and the statement will follow implementing the previous lemmas. 
We are interested in estimates on $L^r$ norms with $r\in (1, \infty)$ and since the operator $\mathbb{P}:L^r\to L^r$ is bounded for such values of $r$, we forget about the operator $\mathbb{P}$ in what follows.
The terms 
$(R_i \cdot \nabla) R_j$ behaves better, thus the worst contribution is given by the first four summand.
The analysis of each of them is essentially the same, and so, to fix the ideas, we will focus on
$$ 
\lambda e^{-t \lambda^2} ( R_2 \cdot \nabla ) ( \phi_L    B_{\lambda} ). 
$$
We rewrite it explicitly as
$$ 
\lambda^2 e^{-t \lambda^2}  \left( [e^{t\Delta} , \phi_L]  B_{\lambda}   \cdot \nabla \right)(  \phi_L    B_{\lambda}),  
$$
and we see that the worst contribution comes when the $\nabla$ falls on $B_{\lambda}$ (otherwise we have an extra $L^{-1}$ gain). Thus, we reduce to estimate
$$ 
\lambda^2 e^{-t \lambda^2} \phi_L  ( [e^{t\Delta} , \phi_L]  B_{\lambda}   \cdot \nabla )  B_{\lambda}.  
$$
We first note (see \eqref{NablaEstimateLambda})
\begin{equation}\label{BoundPartialBeltrami}
\| \nabla B_{\lambda} \|_{L^{\infty}} \lesssim 
\lambda \| B_{\lambda} \|_{L^{\infty}} \leq \lambda.
\end{equation}
Now, using this observation and Lemma \ref{Lemma:dec} with $p = r$, we obtain the following estimates. For short times $0\leq t \leq L^{\varepsilon}$ it holds
\begin{align*}
\| \lambda^2 e^{-t \lambda^2}  \phi_L  ( [e^{t\Delta} , \phi_L]  B_{\lambda}   \cdot \nabla )      B_{\lambda} \|_{L^{r}}
&\lesssim \lambda^2 \| [e^{t\Delta} , \phi_L]  B_{\lambda} \|_{L^{r}} \| \phi_L \|_{L^{\infty}} 
\| \nabla B_{\lambda} \|_{L^{\infty}}\\
&\lesssim \lambda^{3} R L^{-1 + \frac{\varepsilon}{2}+\frac3r - \frac3q    } ,
 \end{align*}
and taking $L \gg \lambda, R$, we use a factor $L^{-\frac\e2}$ to absorb $\lambda^3 R$ and we get the claimed estimate. For large times $t > L^{\varepsilon}$ the estimate is immediate taking advantage of the exponential factor $e^{-t \lambda^2}$. We can indeed estimate everything in $L^{r}$ and take $L \gg \lambda + \lambda^{-1} + R$ without even using the commutator structure. Alternatively we can simply use Lemma \ref{Lemma:dec} again. Now it should be clear that all the other contributions can be estimated in the same way, possibly using Lemma \ref{LemmaT1} instead of Lemma \ref{Lemma:dec} (in which case the proof is even easier, like for $R_1$ and $R_4$).

We will leave the analysis to the reader, after doing two final remarks. The first remark concerns the contribution $ \phi_L  ( B_{\lambda}  \cdot \nabla ) R_j$. This behaves exactly as the contributions $(R_j \cdot \nabla )  \phi_L    B_{\lambda}$ except for the fact that one has to take into account the partial derivatives of the $R_j$, $j=2,3$. This is harmless since it only introduces (in the worst case) an extra factor $\lambda$ that can be absorbed by an extra $L^{\varepsilon}$; see Remark \ref{PartialDerivativesRemark}. 
   
The second remark concerns the contributions $(R_i \cdot \nabla) R_j$. They behave much better than the others in the small time regime ($t \leq L^{\varepsilon}$), but they have a worst decay for long times (polynomial instead of exponential). Indeed, until now we have not really used the Lemmas \ref{LemmaT1}-\ref{Lemma:dec} in the large time regime ($t > L^{\varepsilon}$) and we actually need them when looking at the long time behavior of $ (R_i \cdot \nabla) R_j$, $i,j =2,3$. By using Lemmas \ref{LemmaT1} and/or \ref{Lemma:dec} with $p= 2r$ for both $R_i$ and $R_j$ (the action of $\nabla$ is handled as explained in Remark \ref{PartialDerivativesRemark}) that gives a decay $t^{- 3/2+\frac{3}{4r}} t^{- 3/2+\frac{3}{4r}} = t^{-3+\frac{3}{2r}}$, as claimed.
\end{proof}

\begin{rem}
We now make some comments on the relationship between $r$ and $q$ in \eqref{relazione r q}.
\begin{itemize}
\item Since $q>3$, in order to get smallness, the equation \eqref{relazione r q} imposes the restriction $r>3/2$.
\item If $r\geq q$, the equation \eqref{relazione r q} is always satisfied.
\item For $0<\delta\ll 1$ small enough, one can consider $0<\e\leq \delta$ such that, for any $3<q < \infty$
$$
-1  + \frac{3}{3-\delta} - \frac{3}{q} + \varepsilon \leq - c \delta.
$$
This allow us to apply Lemma \ref{ForcingTermLemma} to get the following estimate
\begin{equation}\label{CommBound3}
\| \mathbb{P} \, \Omega_{u_0}(t,\cdot) \|_{L^{3 - \delta}} \lesssim 
\left\{ \begin{array}{ll}
t^{-5/2 + C \delta} L^{- \theta s} & \mbox{for $t > L^{\varepsilon}$}, \\
L^{- c \delta} & \mbox{for $t \leq L^{\varepsilon}$} ,
\end{array}
\right. 
\end{equation}\
where $\theta \in (0,1)$ and and $C >1$ are absolute constants and $c >0$ only depends on $q$.
\end{itemize}
\end{rem}

Finally, we conclude this subsection proving that our data are large in any Navier-Stokes critical space.

\begin{lem}\label{SizeLemmaM}
Let $q \in (3, \infty)$ and $B_{\lambda} \in L^q \cap L^{\infty}(\R^3)$ be a Beltrami field with frequency $\lambda$, $\| B_{\lambda}\|_{L^{\infty}} =1$ and $\| B_{\lambda}\|_{L^{q}} =R$, and let $\phi \in W^{s,1} \cap W^{s, \infty}(\R^3)$ be a positive smooth function. Let $M>0$ and define the divergence-free vector field 
$$
u_0=M\curl(\phi_L B_\lambda).
$$ % of order $s$ as in Definition \ref{Def:sLocalizing}. 
Then, for all sufficiently large values of $s$ and of the dilation factor $L$ we we have that
\begin{equation}\label{LargeDataM}
M \lesssim \| u_0 \|_{\dot{B}^{-1}_{\infty, \infty}} 
\leq 
\| u_0 \|_{\dot{B}^{-1}_{\infty, 2}}
\lesssim M.
\end{equation}
\end{lem}

\begin{proof}
By rescaling it suffices to consider $M=1$. We will prove the statement when the frequency $\lambda$ of the Beltrami field is strictly positive, the case $\lambda <0$ requires obvious modifications.

We recall that (see \eqref{RecThis1}): %As in the proof of Lemma \ref{} we compute the evolution
$$
e^{t\Delta}u_0 
 = \lambda e^{-t \lambda^2} \phi_L    B_{\lambda}  +  \lambda [e^{t\Delta} , \phi_L]  B_{\lambda}   
+ e^{t\Delta} ( \nabla\phi_L\wedge B_{\lambda} ).
$$
We will see that the first term on the right hand side dominates the others. We first compute the $\dot{B}^{-1}_{\infty,\infty}$ norm: we recall the caloric characterization 
$$
\| f \|_{\dot{B}^{-1}_{\infty,\infty}}:= \sup_{t >0} \sqrt{t} \| e^{t\Delta} f\|_{L^\infty}.
$$
We then compute
\begin{equation}\label{WRTTinKT}
\sup_{t >0} \sqrt{t} e^{-t \lambda^2} \| \phi_L    B_{\lambda}  \|_{L^{\infty}} \gtrsim
\sup_{t >0} \sqrt{t} e^{-t \lambda^2} \phi_L (x^*)   |B_{\lambda} (x^*)| \gtrsim
\sup_{t >0} \sqrt{t}    e^{-t \lambda^2}  \gtrsim \frac{1}{\lambda},
\end{equation}
where we have chosen $x^*$ such that $\| B_{\lambda} \|_{L^{\infty}} \geq |B_{\lambda}(x^*)| $
and $L$ sufficiently large that $| \phi_L(x^*)| > 1/2$.

Then using Lemmas \ref{LemmaT1} and \ref{Lemma:dec} with $p=q=\infty$ we see that the remainders 
are negligible w.r.t. \eqref{WRTTinKT}. This proves the first inequality in \eqref{LargeDataM}. 

For the last inequality in \eqref{LargeDataM} it is enough to use the caloric characterization
$$
\| f \|_{\dot{B}^{-1}_{\infty, 2}} \simeq 
 \left( \int_{0}^{\infty}  \| e^{t\Delta} f\|^2_{L^\infty} \, dt \right)^{1/2}.
 $$
Then, from the decomposition \eqref{RecThis1} together with Lemmas \ref{LemmaT1} and \ref{Lemma:dec} with $p=q=\infty$, the statement easily follows.
\end{proof}

\subsection{Proof of Theorem 1}\label{Sec: Proof of Main Theorem 1}
We can now prove that initial data of the form
$$
u_0=M\curl(\phi_L B_\lambda),
$$
give rise to strong finite energy solutions. To do so, we use Theorem \ref{teo:Chemin-Gallagher} combined with Lemma \ref{ForcingTermLemma}. We provide again the statement of {\bf Theorem 1} for the reader's convenience.
\begin{mainthm*}
Let $M > 0$ and $q \in (3, \infty)$. Let $B_{\lambda} \in L^{q}\cap L^{\infty}(\R^3)$ be a Beltrami field of 
frequency~$\lambda \neq 0$ and such that $\|B_\lambda\|_{L^\infty}=1$. Let  %$\|B_\lambda\|_{L^q}=R$, 
 $\phi \in W^{s,1} \cap W^{s, \infty}(\R^3)$ be a positive function. %a localizing function of order $s$ (see Definition \ref{Def:sLocalizing} below).
We consider the divergence free vector field
\begin{equation}\label{DefW0Bis}
u_0 := M  \curl(\phi_{L} B_{N}) , \quad \phi_{L}(\cdot) := \phi\left( \frac{\cdot}{L} \right), \qquad M, L > 0.
\end{equation}
Then % $u_0\in H^k(\R^3)$ for all $k=0,...,s-1$ and 
for all sufficiently large values of the regularity $s$ and of 
the dilation parameter $L$ the following holds:
\begin{itemize}
\item every scaling invariant norm of $u_0$ is large, namely $\|u_0\|_{\dot{B}^{-1}_{\infty,\infty}}\sim M $;
\item there exists a unique global strong solution $u$ of the Navier-Stokes equation with initial datum $u_0$.
\end{itemize} 
The solution is smooth and has finite energy, in fact satisfies the 
energy identity 
\begin{equation}\label{EnergyId}
\int_{\R^3}|u(t,x)|^2 \, \de x + \int_{0}^{t} \int_{\R^3} |\nabla u(s,x)|^2 \, \de x \de s = 
\int_{\R^3}|u_0(x)|^2 \, \de x .
\end{equation}
\end{mainthm*}
\begin{proof}
First of all, the size of the initial datum follows from Lemma \ref{SizeLemmaM}.
Thus, in order to achieve the proof we only need to show that the condition \eqref{SmallnessConditionCG} is satisfied. Notice that, since $u_0\in  H^{1}(\R^3)$, the smoothness of the solution follows by the classical Prodi-Serrin conditions (in fact $u\in L^4_tL^6_x$).
First, from the inequality \eqref{LargeDataM} we know that $\| u_0 \|_{\dot{B}^{-1}_{\infty, 2}}
\lesssim M$.
Thus, it is sufficient to show that
\begin{equation}\label{SmallnessConditionBis}
\| \mathbb{P} \, \Omega_{u_0} \|_{E} 
\leq \frac{1}{C^*} 
\exp\left(-C^*  M^4  \right).
\end{equation}
To compute this norm we distinguish small and large frequencies (in the definition \eqref{Def:OmegaW0}) and small and large times. \\ 
For high frequencies and large times $t > L^{\varepsilon}$ by Bernstein inequality \eqref{BernsteinIneq1}
\begin{equation}\label{correzione review}
\| \Delta_j \mathbb{P} \, \Omega_{u_0} (t,\cdot)\|_{L^\infty} \lesssim 2^{j \frac{3}{r}}\|  \mathbb{P} \, \Omega_{u_0} (t,\cdot)\|_{L^{r}},
\end{equation}
with $r>3$, and then we use \eqref{CommBound2} to get for some $\theta >0$ that
\begin{align*}
 \sum_{ j \geq 0} 2^{- j} \int_{L^{\varepsilon}}^{\infty}  
&  \| \Delta_j \mathbb{P} \, \Omega_{u_0}(t,\cdot) \|_{L^\infty}    \,dt + 
\sum_{ j \geq 0} 2^{-j} \left( \int_{L^{\varepsilon}}^{\infty} \| \Delta_j \mathbb{P} \, \Omega_{u_0} (t,\cdot)\|^2_{L^{\infty}} t \, dt \right)^{1/2}
\\ 
&\lesssim\sum_{ j \geq 0} 2^{- j+\frac3r j} \int_{L^{\varepsilon}}^{\infty}  
 \| \mathbb{P} \, \Omega_{u_0}(t,\cdot) \|_{L^r}    \,dt + 
\sum_{ j \geq 0} 2^{-j+\frac3r j} \left( \int_{L^{\varepsilon}}^{\infty} \|  \mathbb{P} \, \Omega_{u_0} (t,\cdot)\|^2_{L^{r}} t \, dt \right)^{1/2}
\\
&
\lesssim L^{- \theta s} \sum_{ j \geq 0}   2^{- j+\frac3r j} \int_{L^{\varepsilon}}^{\infty} t^{-3+\frac{3}{2r}}   dt + 
L^{- \theta s} \sum_{ j \geq 0} 2^{-j+\frac3r j} \left( \int_{L^{\varepsilon}}^{\infty} t^{-5+\frac3r}  \, dt \right)^{1/2}
\\ 
&
\lesssim L^{- \theta s - \left(2-\frac{3}{2r}\right)\varepsilon} \sum_{ j \geq 0}   2^{- j+\frac3r j} 
\lesssim L^{- \theta s - \left(2-\frac{3}{2r}\right)\varepsilon} \ll \frac{1}{C^*} 
\exp\left(-C^*  M^4  \right), 
\end{align*}
where in the last inequality we consider $L \gg M$ and that the sum is finite if $r>3$. This choice of $L$ will be implicitly  assumed until the end of the proof.\\
\\
For small times we use again Bernstein inequality \eqref{BernsteinIneq1}
\begin{equation}\label{correzione review 2}
\| \Delta_j \mathbb{P} \, \Omega_{u_0} (t,\cdot)\|_{L^\infty} \lesssim 2^{j \frac{3}{r}}\|  \mathbb{P} \, \Omega_{u_0} (t,\cdot)\|_{L^{r}},
\end{equation} 
together with \eqref{CommBound2} for some $r>2q$ and we get that
\begin{align*}
 \sum_{ j \geq 0} 2^{- j} \int_0^{L^{\varepsilon}}  
&  \| \Delta_j \mathbb{P} \, \Omega_{u_0}(t,\cdot) \|_{L^\infty}   dt + 
\sum_{ j \geq 0} 2^{-j} \left( \int_0^{L^{\varepsilon}} \| \Delta_j \mathbb{P} \, \Omega_{u_0}(t,\cdot) \|^2_{L^{\infty}} t \, dt \right)^{1/2}
\\ 
&
\sum_{ j \geq 0} 2^{- j+\frac3r j} \int_0^{L^{\varepsilon}} \| \mathbb{P} \, \Omega_{u_0}(t,\cdot) \|_{L^r}   dt + 
\sum_{ j \geq 0} 2^{-j+\frac3r j} \left( \int_0^{L^{\varepsilon}} \| \mathbb{P} \, \Omega_{u_0}(t,\cdot) \|^2_{L^{r}} t \, dt \right)^{1/2}
\\
&
\lesssim L^{-1-\frac{3}{2q} + \varepsilon} \sum_{ j \geq 0}   2^{- j+\frac3r j} \int_0^{L^{\varepsilon}}  1 \, dt + 
L^{-1-\frac{3}{q} + \varepsilon} \sum_{ j \geq 0} 2^{-j+\frac3r j} \left( \int_0^{L^{\varepsilon}} t  \, dt \right)^{1/2}
\\ 
&
\lesssim L^{-1-\frac{3}{2q} + 2\varepsilon }  \sum_{ j \geq 0}   2^{- j+\frac3r j}
\lesssim L^{-1-\frac{3}{2q} + 2\varepsilon } \ll \frac{1}{C^*} 
\exp\left(-C^*  M^4  \right).
\end{align*}
To handle small frequencies $j < 0$ we first notice that we have 
\begin{equation}\label{fdhjskduhgudks}
\| \Delta_j \mathbb{P} \, \Omega_{u_0} (t,\cdot)\|_{L^\infty} \lesssim 2^{j \frac{3}{3 - \delta}}
\|  \mathbb{P} \, \Omega_{u_0} (t,\cdot)\|_{L^{3-\delta}},
\end{equation}
by Bernstein inequality \eqref{BernsteinIneq1}. 
Thus for large times $t > L^{\varepsilon}$ 
we can use \eqref{fdhjskduhgudks} and, if we consider $0<\e\leq \delta$, we can apply \eqref{CommBound3} to get: %and we get  for some $\theta >0$:
\begin{align*}
 \sum_{ j < 0 } 2^{- j} \int_{L^{\varepsilon}}^{\infty}  
&  \| \Delta_j \mathbb{P} \, \Omega_{u_0}(t,\cdot) \|_{L^\infty}   dt + 
\sum_{ j < 0} 2^{-j} \left( \int_{L^{\varepsilon}}^{\infty} \| \Delta_j \mathbb{P} \, \Omega_{u_0} (t,\cdot)\|^2_{L^{\infty}} t \, dt \right)^{1/2}
\\ 
&
\lesssim  L^{- \theta s} \sum_{ j < 0}   2^{-j + j\frac{3}{3-\delta}} 
\int_{L^{\varepsilon}}^{\infty} t^{-5/2 + C \delta}   dt + 
L^{- \theta s} \sum_{ j < 0 } 2^{-j + j\frac{\delta}{3-\delta}} \left( \int_{L^{\varepsilon}}^{\infty} t^{-4 + C \delta}  \, dt \right)^{1/2}
\\ 
&
\lesssim L^{- \theta s - \frac32 \varepsilon+C\delta} \sum_{ j < 0}   2^{j\frac{\delta}{3-\delta}}
\lesssim_{\delta} L^{- \theta s - \frac32 \varepsilon+C\delta} \ll \frac{1}{C^*} 
\exp\left(-C^*  M^4  \right).
\end{align*}
Finally, for small frequencies and small times, we use again \eqref{fdhjskduhgudks} and \eqref{CommBound3} and we get 
\begin{align*}
 \sum_{ j < 0} 2^{- j} \int_0^{L^{\varepsilon}}  
&  \| \Delta_j \mathbb{P} \, \Omega_{u_0} (t,\cdot)\|_{L^\infty}   dt + 
\sum_{ j < 0} 2^{-j} \left( \int_0^{L^{\varepsilon}} \| \Delta_j \mathbb{P} \, \Omega_{u_0}(t,\cdot) \|^2_{L^{\infty}} t \, dt \right)^{1/2}
\\ 
&
\lesssim L^{- c \delta} \sum_{ j < 0}   2^{j\frac{\delta}{3-\delta}} \int_0^{L^{\varepsilon}} 1  \, dt + 
L^{- c \delta} \sum_{ j < 0} 2^{j\frac{\delta}{3-\delta}} \left( \int_0^{L^{\varepsilon}} t  \, dt \right)^{1/2}
\\ 
&
\lesssim L^{-c \delta + \varepsilon }  \sum_{ j < 0}   2^{j\frac{\delta}{3-\delta}}
\lesssim_\delta  L^{- c \delta +  \varepsilon} \ll \frac{1}{C^*} 
\exp\left(-C^*  M^4  \right),
\end{align*}
taking $\varepsilon < \frac{c}{2} \delta$, that concludes the proof.
\end{proof}

\section{Vortex reconnection} 
In this section we move our attention to the problem of the vortex reconnetion. We first provide the building blocks for the construction of the initial data and then we will prove {\bf Theorem 2}. 
\subsection{Construction of $u_0^1$} First of all, we define the Beltrami field $B_N$ as follows
\begin{equation}\label{def:BN}
B_{N}(x_1,x_2,x_3):=(\sin N x_3,\cos N x_3, 0).
\end{equation}
It is immediate to check that $B_{N}$ has no zeros.
Moreover, we consider the scalar smooth function $\phi:\R^3\to\R$ defined as
\begin{equation}\label{Rec1Size}
\phi(x):=\frac{1}{(1+|x|^2)^\alpha}, \qquad \alpha \geq 1.
\end{equation}
Note that the function $\phi$ is localized around the origin with high precision when we take $\alpha$ very large.
Lastly, we define
\begin{align}
u_0^{1} &:= \curl(\phi B_{N})=N\phi B_{N}+\nabla\phi\wedge B_{N},\label{def:u01-1}\\
\omega_0^{1}&:=\curl u_0^{1}=\curl\curl(\phi B_{N})\nonumber\\
&=N^2\phi B_{N}+ N\nabla\phi \wedge B_{N}+(B_{N}\cdot\nabla)\nabla\phi-\Delta\phi B_{N}-(\nabla\phi\cdot\nabla)B_{N}.\label{def:omega01}
\end{align}
For simplicity, we also define $\Omega:=\nabla\phi\wedge B_{N}$ and then we can rewrite $u_0^1,\omega_0^1$ as follows
\begin{align}
u_0^{1} &=N\phi B_{N}+\Omega,\\
\omega_0^{1}&=N^2\phi B_{N} + N \Omega+\curl\Omega.\label{def:omega01-2}
\end{align}
It is worth to note that the function $\phi$ has been chosen so that $u_0^{1}$ is a ``regular perturbation" of its principal part $N \phi B_N$. % and $N^2 \phi B_N$, respectively. 
%A similar issue will be encountered for $W$ and we will proceed in a similar way, but with a different choice of the localizing function $\psi$.
We now prove two elementary properties of $\omega_0^1$ and $\phi$.
\begin{prop}
Let $\omega_0^1$ be as in \eqref{def:omega01-2} and let $N \gg \alpha$. Then, $\omega_0^1$ does not have any zero.
\end{prop}
\begin{proof}
Since $\phi$ is a strictly positive function, the vector field $\phi B_{N}$ is always different from $0$. Now, if $\alpha \ll N$ it is not hard to show that the vector field
$\Omega$ can be considered as a small perturbation of $N\phi B_{N}$.
This is essentially due to the decay properties of the derivatives of $\phi$, namely to the fact that $\nabla^k \phi$ decays faster then $\nabla^m \phi$ if $k> m$. Thus, since by definition the function $\Omega$ involves one or two derivatives of $\phi$, we obtain the following point-wise bound
\begin{equation}
|\omega_0^1(x)|\geq N^2|\phi(x) B_{N}(x)|-|N \Omega(x)+\curl\Omega(x)|>\frac{c N^2}{(1+|x|^2)^\alpha},\label{est:lb-omega1}
\end{equation}
that is valid for some small constant $0< c <1$ as long as $N \gg \alpha$. The inequality \eqref{est:lb-omega1} guarantees that the vector field $\omega_0^1$ does not have any zero.
\end{proof}

\begin{prop}\label{pro:decayphi}
Let $\phi$ be as in \eqref{Rec1Size} and denote by $\widehat{\phi}$ its Fourier transform. Then, for any $C_1>0$ there exists $C_2>0$ such that for every $n\in \N$ with $n\leq C_1\alpha$ the following point-wise bound holds
\begin{equation}\label{est:point-psia}
|\widehat{\phi}(\xi)|\leq C_2 \frac{\alpha^n}{|\xi|^n}.
\end{equation}
\end{prop}
\begin{proof}
The bound easily follows using the regularity of the function $\phi$:
%We use the regularity of $\phi$ to obtain a pointwise upper bound on $\widehat{\phi}$: 
from the definition of Fourier transform and by integrating by parts $n$ times%, with $n\in \N$ as big as we want, we obtain that
\begin{equation}
|\widehat{\phi}(\xi)|=\left|\int_{\R^3} e^{i\xi\cdot x}\phi(x)\de x\right|\leq \frac{1}{|\xi|^n}\int_{\R^3} |\nabla^n\phi(x)|\de  x 
\leq C_2  \frac{\alpha^n}{|\xi|^n}, \qquad n \leq C_1 \alpha.
\end{equation}
Here we used that for $n \leq C_1\alpha$ one has
$$
 \alpha (\alpha +1) \dots {(\alpha +n)} \leq C_2 \alpha^n,
$$
and this concludes the proof.
%In particular, using that $\nabla^k \phi$ decays faster then $\nabla^m \phi$ if $k> m$, we can consider a constant $C$ which is independent on $n$.
\end{proof}

\begin{prop}
Let $B_N$ and $\phi$ be as in \eqref{def:BN} and \eqref{Rec1Size} respectively. Then, for every $C>0$ and every $n\in\N$ with $n\leq C\alpha$, the following point-wise upper-bound holds
\begin{equation}
|\widehat{\phi B_{N}}(\xi)|\lesssim_n \frac{\alpha^n}{\sqrt{|\xi_1|^2+|\xi_2|^2+|\xi_3+N|^2}^{\,n}} +
\frac{\alpha^n}{\sqrt{|\xi_1|^2+|\xi_2|^2+|\xi_3-N|^2}^{\,n}}, \quad n \leq C\alpha.
\end{equation}
\end{prop}
\begin{proof}
First of all, notice that (up to a constant)
\begin{equation}\label{DecB}
\widehat{B}_{N}(\xi_1,\xi_2,\xi_3)  = \sum_{\ell = 1}^{4} \widehat{B_{N}^{\ell}}(\xi_1,\xi_2,\xi_3),
\end{equation}
where
$$
\widehat{B_{N}^{1}} := \frac12 (\delta_0(\xi_1)\otimes\delta_0(\xi_2)\otimes\delta_0(\xi_3 - N), 
0,0), 
\qquad 
\widehat{B_{N}^{2}} := \frac12 (\delta_0(\xi_1)\otimes\delta_0(\xi_2)\otimes\delta_0(\xi_3 +  N),0,0),
$$

$$
\widehat{B_{N}^{3}} := \frac1{2i} (0, 
\delta_0(\xi_1)\otimes\delta_0(\xi_2)\otimes\delta_0(\xi_3 - N),0), 
\qquad 
\widehat{B_{N}^{4}} := -  \frac1{2i} (0, 
\delta_0(\xi_1)\otimes\delta_0(\xi_2)\otimes\delta_0(\xi_3 +  N),0).
$$
and $\delta_0$ is the delta distribution centered in $0$. For definiteness we will only analyze the 
contribution $\widehat{B_{N}^{1}}$. The analysis of the other 
terms requires obvious modifications. 
We have %Thus, using the above computations, we have that
$$
\widehat{\phi B_{N}^{1}}(\xi)=\widehat{\phi}* \widehat{B_{N}^{1}}(\xi)
= (\widehat{\phi}(\xi_1,\xi_2,\xi_3-N),0,0),
%=(\widehat{\psi}_\alpha(\xi_1,\xi_2,\xi_3-N),\widehat{\psi}_\alpha(\xi_1,\xi_2,\xi_3-N),0).
$$
and by using Proposition \ref{pro:decayphi} we obtain that
\begin{equation}
|\widehat{\phi B_{N}^{1}}(\xi)|\lesssim \frac{\alpha^n}{\sqrt{|\xi_1|^2+|\xi_2|^2+|\xi_3-N|^2}^{\,n}}, \quad n \leq C\alpha, 
\end{equation}
so the statement follows.
\end{proof}

The above estimate will be useful while dealing with high-frequencies in the estimate of the $H^r$ norm of $\omega_0^1$.
\begin{prop}\label{prop:sizeu01}
Let $u_0^1$ as in \eqref{def:u01-1} and assume that $N \gg \alpha$. Then, for every~$r\geq 0$,  $u_0^1\in H^r(\R^3)$ and satisfies
\begin{equation}
\|u_0^1\|_{H^r}\leq C(r) N^{r+1},\qquad \|\omega_0^1\|_{H^r}\leq C(r) N^{r+2}.
\end{equation}
\end{prop}
\begin{proof}
We now compute the $\|\cdot\|_{H^r}$-norm of $u_0^1$. Recalling $u_0^1 := \curl(\phi B_N)$ we have that
\begin{align}
\|u_0^1\|_{H^r}^2&\lesssim \int_{\R^3}(1+|\xi|^2)^r|\xi|^2|\widehat{\phi}(\xi_1,\xi_2,\xi_3-N)|^2\de \xi\nonumber\\
&= \int_{\R^3}(1+z_1^2+z_2^2+(z_3+N)^2)^r(z_1^2+z_2^2+(z_3+N)^2)|\widehat{\phi}(z)|^2\de z\nonumber\\
&\lesssim_r\left(\|\phi\|_{H^{r+1}}^2+N^{2r}\|\phi\|_{H^1}^2+N^2\|\phi\|_{H^r}^2+N^{2(r+1)}\|\phi\|_{L^2}^2\right),\label{norma-u0}
\end{align}
where in the second line we used the change of variables $\xi_3-N=z_3$ while in the third line we used standard inequalities and the definition of $\|\cdot\|_{H^r}$-norm. Moreover, by using \eqref{est:point-psia} it is immediate to show that
\begin{equation}\label{PhiNorm}
\|\phi\|_{H^r}\lesssim_r \alpha^n,
\end{equation}
provided that $n$ is big enough w.r.t. $r$, namely $n>r+2$, and the proof follows from the choice of $\alpha\ll N$. Finally, the estimate on $\omega_0^1$ follows from the fact that
$$
\|\omega_0^1\|_{H^r}\leq \|u_0^1\|_{H^{r+1}},
$$
and this concludes the proof.
\end{proof}

Now, we want to analyze the evolution of $\omega_0^1$ under the heat flow. 
\begin{lem}\label{lem:heat-omega01}
Let $\nu>0$ be fixed and let $\omega_0^1$ be as in \eqref{def:omega01}. Let $F_N(t)$ be defined as
\begin{equation}\label{Def:FN(t)}
F_N(t):=\frac{1}{N}\int_{-\frac{N}{2}}^{\frac{N}{2}}e^{-2\nu t \xi_3^2}\de \xi_3,
\end{equation}
it satisfies $|F_N(t)| \leq 1$.
Moreover, for all $r \in \mathbb{N}$ and for all $n\in \N$ big enough, there exists a constant $C(r,n)>0$ such that the following bound holds
\begin{equation}
\|e^{\nu t\Delta}\omega_0^1\|_{H^r}
\leq 
%C(r)\left(\alpha^{2n}N^{2(r+3)} e^{-\nu N^2 t}+ \frac{\alpha^{2n}}{N^{2(n-r-3)}} F_N(t)\right)
C(r,n)\left(\alpha^{n}N^{r+2} e^{-\nu \frac12 N^2 t}+ \frac{\alpha^{n}}{N^{n-r-4}} \sqrt{F_N(t)}\right).\label{est:finalB1}
\end{equation}
\end{lem}
\begin{proof}
We use the decomposition \eqref{DecB}. As before, the analysis of all the four contribution is the same (with obvious modification), so we will focus on 
$$
e^{\nu t\Delta}(\curl\curl(\phi B_{N}^{1})).
$$
We
define $P$ to be the following Fourier projection operator % variables
\begin{equation}
\widehat{P f}(\xi)=\chi_{\{|\xi_3|\leq N/2\}}\widehat{f}(\xi),
\end{equation}
as similarly done several times before.
We split the action of the heat kernel on $\curl\curl(\phi B_{N}^{1})$ in the following way 
$$
e^{\nu t\Delta}(\curl\curl(\phi B_{N}^{1})
=e^{\nu t\Delta}(\curl\curl((P\phi)B_{N}^{1}))+e^{\nu t\Delta}(\curl\curl((1-P\phi)B_{N}^{1})):=I+II.
$$
In order to estimate $I$ and $II$ we use the formula (hereafter $\mathcal{F}g := \widehat{g}$)
\begin{equation}
\mathcal{F}[\curl\curl v](\xi)=(i\xi)\wedge((i\xi)\wedge\widehat{v}(\xi)),\label{eq:fouriercurl}
\end{equation}
obtaining that
$$
I:=\mathcal{F}^{-1}\left[e^{-\nu t|\xi|^2}\mathcal{F}[\curl\curl((P\phi)B_{N}^{1})]\right]=-\mathcal{F}^{-1}\left[e^{-\nu t|\xi|^2}\xi\wedge\xi\wedge((\chi_{\{|\xi_3|\leq N/2\}}\widehat{\phi})*\widehat{B_{N}^{1}})\right].
$$
By Plancherel's Theorem we have that
\begin{align}
\|I\|_{H^r}^2&\lesssim\int_{\R^3}(1+|\xi|^2)^r|\xi|^4e^{-2\nu t|\xi|^2}|(\chi_{\{|\xi_3|\leq N/2\}}\widehat{\phi})*\widehat{B_{N}^{1}}|^2\de\xi\nonumber\\
&=\int_{A}(1+|\xi|^2)^r|\xi|^4e^{-2\nu t|\xi|^2}|\widehat{\phi}(\xi_1,\xi_2,\xi_3-N)|^2\de \xi,\label{def:Hs-I}
\end{align}
where the set $A$ is defined as
$$
A:=\{ (\xi_1,\xi_2,\xi_3)\in\R^3:|\xi_3-N|\leq N/2 \}.
$$
In particular, if $(\xi_1,\xi_2,\xi_3)\in A$, the third variable satisfies
\begin{equation}
\frac{N}{2}\leq \xi_3\leq \frac{3}{2}N,
\end{equation}
and since the exponential function in \eqref{def:Hs-I} is decreasing, proceeding as in the estimates in \eqref{norma-u0}, we obtain the bound
\begin{equation}\label{Combining1}
\|I\|_{H^r}^2\leq Ce^{-\nu N^2 t}\left(\|\phi\|_{H^{r+2}}^2+N^{2r}\|\phi\|_{H^2}^2+N^4\|\phi\|_{H^r}^2+N^{2(r+2)}\|\phi\|_{L^2}^2\right).
\end{equation}
We now analyze $II$: by its definition we have that
\begin{align}
\|II\|_{H^r}^2&\lesssim \int_{\R^3}(1+|\xi|^2)^r|\xi|^4e^{-2\nu t|\xi|^2}|(\chi_{\{|\xi_3|> N/2\}}\widehat{\phi})*\widehat{B_{N}^{1}}|^2\de\xi\nonumber\\
&=\int_{\R^3\setminus A}(1+|\xi|^2)^r|\xi|^4e^{-2\nu t|\xi|^2}|\widehat{\phi}(\xi_1,\xi_2,\xi_3-N)|^2\de \xi.\nonumber
\end{align}
We split the set $\R^3\setminus A=A_1\cup A_2\cup A_3$, where
$$
A_1:=\left\{\xi\in\R^3:\xi_3>\frac{3}{2}N\right\},\,\,
A_2:=\left\{\xi\in\R^3:-\frac{N}{2}<\xi_3<\frac{N}{2}\right\},\,\,
A_3:=\left\{\xi\in\R^3:\xi_3<-\frac{N}{2}\right\}.
$$
On the sets $A_1$ and $A_3$ we can actually proceed as done for $I$ obtaining that
\begin{equation}\label{Combining2}
\|II\|_{H^r(A_1\cup A_3)}^2\leq C e^{-\nu N^2 t}\left(\|\phi\|_{H^{r+2}}^2+N^{2r}\|\phi\|_{H^2}^2+N^4\|\phi\|_{H^{r+1}}^2+N^{2(r+2)}\|\phi\|_{L^2}^2\right).
\end{equation}
On the set $A_2$ we proceed in a different way. Recalling 
 the point-wise bound \eqref{est:point-psia} we get
\begin{align}\label{Combining3}
\|II\|_{H^r(A_2)}^2&\leq C\alpha^{2n}\iint_{\R^2}\int_{-\frac{N}{2}}^{\frac{N}{2}}\frac{(1+|\xi|^2)^r|\xi|^4}{(|\xi_1|^2+|\xi_2|^2+|\xi_3-N|^2)^n}e^{-2\nu t|\xi|^2}\de \xi\nonumber\\
&\leq C \alpha^{2n}\iint_{\R^2}\frac{(\xi_1^2+\xi_2^2+N^2)^{r+2}}{(\xi_1^2+\xi_2^2+N^2)^n} \de \xi_1\de\xi_2\int_{-\frac{N}{2}}^{\frac{N}{2}}e^{-2\nu t \xi_3^2}\de\xi_3\nonumber\\
&\leq C \frac{\alpha^{2n}}{N^{2(n-r-4)}} F_N(t),
\end{align}
provided that $n> r+4$ and $F_N(t)$ is defined as in \eqref{Def:FN(t)}.
In particular, note that by the dominated convergence theorem, the function $F_N$ satisfies 
$$
\lim_{t\to 0}F_N(t)=1,\qquad\lim_{t\to\infty}F_N(t)=0.
$$
Moreover, if we change variables, $F_N$ can be rewritten in the following form
$$
F_N(t)=\int_{-\frac{1}{2}}^{\frac{1}{2}}e^{-2\nu N^2 t \xi_3^2}\de \xi_3.
$$
Thus, if $n> r+4$ we also have that $\|II\|_{H^r(A_2)}^2$ can be made as small as we want for $N$ big enough. 
Combining \eqref{Combining1}-\eqref{Combining3} (and using %\eqref{PhiNorm} with $n=\alpha +2$
to estimate the r.h.s. of \eqref{Combining1}-\eqref{Combining2})
we arrive to 
\begin{equation}
\|e^{\nu t\Delta}\curl\curl(\phi B_{N}^{1})\|_{H^r}^2\leq C \left(\alpha^{2n}N^{2(r+2)} e^{-\nu N^2 t}+ \frac{\alpha^{2(r+4)}}{N^{2(n-r-4)}} F_N(t)\right),
\end{equation}
that concludes the proof.
\end{proof}

\subsection{Construction of $u_0^2$}
We now continue with the construction of the second vector field. We define the vector field $W$ as follows:
\begin{equation}\label{def:WM}
W(x_1,x_2,x_3):=(\sin(x_2),\sin(x_3),\sin(x_1)).
\end{equation}
%thus we have denoted with $M$ the frequency of oscillations of $W$. 
Notice that $W$ is not a Beltrami field, but it is an eigenvector of the double $\curl$ operator, i.e.
\begin{equation}\label{BoubleCurlEq1}
\curl\curl W= W.
\end{equation}
$W$ is divergence-free and we will see that it has at least one hyperbolic zero, which in particular is stable under regular perturbations. %and, due to the incompressibility condition, it must be a saddle point.

We now fix a target time $T>0$ and we define the Schwarz function
$\psi:\R^3\to\R$ as 
\begin{equation}\label{def:psi}
\psi(x_1,x_2,x_3)=e^{-\frac{|x|^2}{8\nu T}},
\end{equation}
whose Fourier transform is given by
$$
\widehat{\psi}(\xi)=8(\nu T)^{\frac{3}{2}}e^{-2\nu |\xi|^2T}.
$$

\begin{prop}\label{prop:punto critico}
Let $W$ and $\psi$ be defined as in \eqref{def:WM} and \eqref{def:psi} respectively. Then, $(0,0,0)$ is a hyperbolic zero of the vector field $\curl\curl(\psi W)$.
\end{prop}
\begin{proof}
By definition we have that
$$
\curl\curl(\psi W)=\psi W +  \nabla\psi \wedge \curl W+(W\cdot\nabla)\nabla\psi-\Delta\psi W-(\nabla\psi\cdot\nabla)W.
$$
Thus, since $W(0)=0$, $\psi(0)=1$ and $\nabla\psi(0)=0$ we also have %A straightforward computation shows that 
$$
\curl\curl(\psi W) \big|_{x=0} = 0. \qquad
$$
Moreover, a straightforward computation gives
\begin{equation}
\nabla(\curl\curl(\psi W) )|_{x=0}=\frac{1}{\nu T}
\begin{pmatrix}
0 & 1+\nu T & -\frac{1}{4}\\
-\frac{1}{4} & 0 & 1+\nu T\\
1+\nu T & -\frac{1}{4} & 0
\end{pmatrix}.
\end{equation}
Thus, for the Jacobian determinant we have that
\begin{equation}
\det \nabla(\curl\curl(\psi W) )|_{x=0}=\frac{1}{\nu T}\left[(1+\nu T)^3-\frac{1}{64}\right]\neq 0,
\end{equation}
and the eigenvalues are
\begin{align}
\lambda_1&=1+\frac{3}{4\nu T}\\
\lambda_2&=-\frac{(3+4\nu T)+(5+4\nu T)\sqrt{3}i}{8\nu T}\\
\lambda_3&=\bar\lambda_2.
\end{align}
Then, we can conclude that $x=0$ is a hyperbolic zero. 
\end{proof}

Let $T>0$ be given. We define our second reference vector field as follows
\begin{align}
u_0^2&:=e^{-\nu T\Delta}\curl(\psi W)\label{def:u02},\\
\omega_0^2 &:= \curl u_0^2 = e^{-\nu T\Delta}\curl\curl(\psi W)\label{def:omega2},
\end{align}
where in \eqref{def:omega2} we have simply used the fact that the heat operator $e^{-\nu T\Delta}$ commutes with the $\curl$. 
Moreover, it is worth to point out that the choice of $\psi$ as in \eqref{def:psi} has been made in order to have a control on the $H^r$ norm of $\omega_0^2$. Roughly speaking the inverse heat operator $e^{-T\Delta}$ grows exponentially at high frequencies but $\widehat{\psi}$ and $W$ are localized at low frequencies (of size $\sim 1$), thus we are able to obtain a control of the growth of the Sobolev norms. The following lemma holds.

\begin{lem}\label{lem:stima-omega02}
Let $T>0$ be given and let $\omega_0^2$ be as in \eqref{def:omega2}. Then, for every $r\geq 0$, there exists a constant $C:=C(T,r,\nu)>0$ such that
\begin{equation}\label{Rec2Size}
\|\omega_0^2\|_{H^r}\leq C(T,r,\nu)e^{C(1+\nu T)}.
\end{equation}
Moreover, $(0,0,0)$ is a hyperbolic zero of the vector field $e^{\nu T\Delta}\omega_0^2$.
\end{lem}
\begin{proof}
The core of the argument is that its Fourier transform of $W$ is given by 2 deltas centered in $0$ times a delta 
centered in $1$, for example
\begin{equation}\label{eq:fouriercurl2}
2i \left(\widehat{W}(\xi_1,\xi_2,\xi_3)\right)_1=\delta(\xi_1)\otimes\delta(\xi_2-1)\otimes\delta(\xi_3)
- \delta(\xi_1)\otimes\delta(\xi_2+1)\otimes\delta(\xi_3).
\end{equation} 
By using \eqref{eq:fouriercurl}, \eqref{eq:fouriercurl2} and the symmetry 
$\xi_2 \leftrightarrow - \xi_2$ we can estimate 
\begin{align}
\|(\omega_0^2)_1\|_{H^r}^2 &\leq\int_{\R^3}e^{2\nu T|\xi|^2}|\xi|^4(1+|\xi|^2)^r|\widehat{\psi}*\widehat{W}|^2(\xi)\de \xi\nonumber\\
&= \int_{\R^3}e^{2\nu T|\xi|^2}|\xi|^4(1+|\xi|^2)^r|\widehat{\psi}(\xi_1,\xi_2-1,\xi_3)|^2\de \xi\nonumber\\
&=8(\nu T)^{\frac{3}{2}}\int_{\R^3}e^{2\nu T|\xi|^2}|\xi|^4(1+|\xi|^2)^r e^{-4\nu\xi_1^2T}e^{-4\nu|\xi_2-1|^2T}e^{-4\nu\xi_3^2T}\de \xi\nonumber\\
&=8(\nu T)^{\frac{3}{2}} e^{4\nu T}\int_{\R^3}|\xi|^4(1+|\xi|^2)^r e^{-2\nu\xi_1^2T}e^{-2\nu|\xi_2-2|^2T}e^{-2\nu\xi_3^2T}\de \xi,\label{eq:inv-heat-omega2}
\end{align}
where in the last line we have used that 
$$
2 \xi_2^2 -4|\xi_2-1|^2 = 4 - 2|\xi_2-2|^2.
$$
%$$
%\xi_2^2+2M^2-4\xi_2M=|\xi_2-2M|^2-2M^2.
%$$
Note that the integral in \eqref{eq:inv-heat-omega2} is bounded by an absolute constant. Then, by symmetry, it is immediate to check that
\begin{equation}
\|\omega_0^2\|_{H^r}\leq C(T,\nu)e^{C(1+\nu T)},
\end{equation}
for some large constant $C$ that only depends on $r$ and $\nu$. % (we omit this dependence in the notations).

Finally, recalling the definition \eqref{def:omega2}, if we apply the heat operator to $\omega_0^2$ we get that 
$$
e^{\nu T\Delta}\omega_0^2=\curl\curl(\psi W),
$$
and the conclusion follows from Proposition \ref{prop:punto critico}.
\end{proof}

\subsection{Proof of Theorem 2}
We can now prove our second main theorem. We enunciate it again below for the reader's convenience.
\begin{mainthmdue*}
Given any constants $\nu$, $T>0$, there exists a (small) smooth divergence-free vector field $u_0:\R^3\to\R^3$ such that \eqref{eq:ns} admits a unique global smooth solution $u$ with initial datum $u_0$, such that the vortex lines at time $t=0$ and $t=T$ are not topologically equivalent, meaning that there is no homeomorphism of $\R^3$ into itself mapping the vortex lines of $u(0,\cdot)$ into that of $u(T,\cdot)$.
\end{mainthmdue*}
\begin{proof}%[Proof of Theorem \ref{thm:main}]
Let $T>0$ be given and define the vector fields $u_0^1, u_0^2$ as follows
\begin{equation}
u_0^1=\curl(\phi B_{N}),
\end{equation}
\begin{equation}
u_0^2= e^{-\nu T\Delta}\curl(\psi W).
\end{equation}
The initial datum of \eqref{eq:ns} will be the vector field $u_0$ defined as
\begin{equation}\label{def:u_0}
u_0=\rho(u_0^1+ u_0^2),
\end{equation}
where $\rho$ and $N$ are a small and large parameter, respectively, to be chosen such that:
\begin{itemize}
\item[$(i)$] the equations \eqref{eq:ns} admit a unique global smooth solution;
\item[$(ii)$] at time $t=0$, $\omega_0^2 := \curl u_0^2$ is a small perturbation of 
$\omega_0^1 := \curl u_0^1$;
\item[$(iii)$] at time $t=T$, the rescaled vector field $\frac{\omega(T,\cdot)}{\rho}$is a small perturbation of $e^{T\Delta}\omega_0^2:=\curl\curl(\psi W)$.
\end{itemize}
Moreover note that, since the heat operator preserves the divergence-free condition, $\dive u_0=0$.\\
\\
\underline{\em Step 1} \,\,Size of the initial datum.\\
\\
First of all, we compute the $H^r$ norm of $u_0$ and $\nabla u_0$. Since
$$
\|u_0^2\|_{H^r}\sim \|\omega_0^2\|_{H^{r-1}},
$$
we use Proposition \ref{prop:sizeu01} 
%{\color{red} usare gli esponenti in accordo con il lemma oppure scegliere $\alpha=1$ cos\`i da evitare potenze strane di $N$?} 
and Lemma \ref{lem:stima-omega02} to obtain
\begin{equation}\label{eq:InitiHr}
\| u_0\|_{H^r}\leq
\rho C(T,r,\nu, \alpha) \left(N^{r+1} +  e^{C (1+\nu T)}  \right) \leq \rho  C(T,r,\nu, \alpha)  N^{r+1}  \ll C_r \nu,
\end{equation}
taking $N$ sufficiently large in the second inequality and $\rho$ sufficiently small in the last one. In this way, Theorem \ref{lem:stima-ns} ensures the global well-posedness of \eqref{eq:ns}. Moreover, if we define $\rho = N^{-\beta}$ with $\beta > 2r + 5$ also have from Lemma \ref{stima norme ns} the following {\it a priori} estimate
\begin{equation}\label{GlobalControl}
\|u(t,\cdot)\|_{H^r} \leq C_r\|u_0\|_{H^r} \leq \rho C  N^{r+1},
\end{equation}
where we have abbreviated $C(T,r,\nu, \alpha)$ to $C$.
%In fact we will choose $N \gg 1$ and $\rho = N^{-\beta}$ with $\beta > 2r + 5$. 
We recall that the value of $r$ has to be chosen to be large enough, in particular in such a way that the $H^{r-1}$ norm dominates the $L^{\infty}$ norm. In fact, we will fix 
$r> 5/2$ (say $r=3$). The reason of these choices will be clarified in the next steps, however
this choice of $\rho$ is clearly enough for \eqref{eq:InitiHr} and \eqref{GlobalControl} to be true.  
\\
\\
\underline{\em Step 2}\,\,\, Integral lines at time $0$.\\
\\
By the definition of $u_0$, the initial vorticity $\omega_0$ is given by
\begin{equation}\label{VorticityFinalDef}
\omega_0=\rho(\omega_0^1+ \omega_0^2).
\end{equation}
%It is worth noticing that%, by passing in Fourier variables we can show that
%$$
%e^{-\nu T\Delta}\omega_0^2:= \curl\curl(\psi W)=\curl e^{-\nu T\Delta}\curl(\psi W).
%$$
Consider the rescaled datum $\tilde{\omega}_0=\rho^{-1}\omega_0$: we have that (recall the estimate \eqref{Rec2Size})
\begin{equation}\label{REscDataumEstimate}
\|\tilde \omega_0 - \omega_0^1\|_{H^r} =  \| \omega_0^2 \|_{H^r} \leq  C e^{ C (1+ \nu T)}  .
\end{equation}
Using this fact it is easy to show that, after choosing $N= N(\nu, T)$ sufficiently large, the vector field $\tilde{\omega}_0$, and so $\omega_0$, does not have zeros.
%\footnote{\textcolor{red}{Indeed we may take $\delta =1$ and choose $N$ large}}. 
We prove this last assertion: let $R>0$, for $|x|\leq R$ one has that
\begin{align*}
|\tilde\omega_0|&\geq |\omega_0^1|-|\tilde\omega_0- \omega_0^1|\\
&\geq \frac{C N^2}{(1+|x|^2)^\alpha} -  e^{C (1+ \nu T)}\\
&\geq \frac{C N^2}{R^{2\alpha}} -   e^{C (1+\nu T)}>0,
\end{align*}
for $N = N(R, \nu, T)$ big enough; the second inequality follows by \eqref{REscDataumEstimate} and Sobolev embedding. On the other hand, if $|x|>R$ and $R$ is sufficiently large (depending only on $\nu,T$), 
we can use that $\psi$ decays faster than $\phi$, show that
\begin{align*}
|\tilde\omega_0|&\geq |\omega_0^1| - |\omega_0^2(x)|\\
& \geq \frac{C N^2}{(1+|x|^2)^\alpha} -   |\omega_0^2(x)| 
%\\ \nonumber
%&
%>  \frac{1}{(1+|x|^2)^\alpha} - C \delta M^2|\omega_0^2(x)| 
>  \frac{1}{(1+|x|^2)^\alpha}, \qquad |x| >R.
% OLD \geq \frac{C N}{(1+|x|^2)^\alpha}-C \delta M^2|e^{-T\Delta}(\psi W)|\\
%&\geq \frac{C N}{(1+|x|^2)^\alpha}-C \delta |x-M|^2 e^{-|x-M|^2}>0.
\end{align*}
In this inequality we simply used the pointwise estimate 
\begin{equation}\label{ExpBeatPoly}
|\omega_0^2(x)| \lesssim \frac{1}{(1+|x|^2)^\alpha}, \qquad |x| >R, 
\end{equation}
that is valid as long as $R$ is sufficiently large (depending only on $\nu,T$). 
Indeed, assuming that $|x_1| \gtrsim |x|$, integrating by parts $k$-times in the $x_1$ direction and %we have 
proceeding as in \eqref{eq:inv-heat-omega2} we can show that
\begin{align}
|(\omega_0^2)_1 (x)| &
\leq \frac{1}{ |x_1|^k} \int_{\R^3}e^{2\nu T|\xi|^2}
| \partial_{\xi_1}^k \mathcal{F}(\curl \curl (\psi W))|(\xi)\de \xi\nonumber\\
&
\leq \frac{1}{|x_1|^k} \int_{\R^3}e^{2\nu T|\xi|^2}
|\xi|^{k+2} |  ( \widehat{\psi}*\widehat{W})|(\xi)\de \xi\nonumber\\
&\leq \frac{C}{|x_1|^k}\int_{\R^3}e^{2\nu T|\xi|^2}|\xi|^{k+2}\widehat{\psi}(\xi_1,\xi_2-1,\xi_3)|^2\de \xi\nonumber\\
&= \frac{C (\nu T)^{3/2}}{|x_1|^k}\int_{\R^3}e^{2\nu T|\xi|^2}|\xi|^{k+2} e^{-4\nu \xi_1^2T}e^{-4\nu|\xi_2-1|^2T}e^{-4\nu \xi_3^2T}\de \xi\nonumber\\
&=\frac{C(\nu T)^{3/2}}{|x_1|^k} e^{4 \nu T}\int_{\R^3}|\xi|^{k+2} e^{-2\nu\xi_1^2T}e^{-2\nu|\xi_2-2|^2T}e^{-2\nu\xi_3^2T}\de \xi,\label{stima-omega02}
\end{align}
and the modulus of the right hand side satisfies the inequality \eqref{ExpBeatPoly}
as long as $k > 2\alpha$ and $R = R(\nu,T) \gg 1$. If we are in the case $|x_j| \gtrsim |x|$ for $j =2$ or $j=3$ 
the same argument, integrating w.r.t. $x_j$ gives the desired result. Moreover, an analogous argument
works for the other components of $\omega_0^2$.\\
\\
\underline{\em Step 3}\,\,\, Evolution under the Navier-Stokes flow.\\
\\
Let $(u,p)$ be the unique strong solution starting from $u_0$ defined in Step 1 and denote with $\omega=\curl u$ its vorticity. By Duhamel's formula $\omega$ satisfies the following identity
\begin{equation}
\omega(t,\cdot)=e^{\nu t\Delta}\omega_0+D(t,\cdot),
\end{equation}
where $D(t,\cdot)$ is defined as
\begin{equation}
D(t,\cdot):=\int_0^t e^{\nu(t-s)\Delta}\dive(\omega(s,\cdot)\otimes u(s,\cdot)-u(s,\cdot)\otimes\omega(s,\cdot))\de s.
\end{equation}
By \eqref{GlobalControl} we have
\begin{equation}
\|u(t,\cdot)\|_{H^r} \leq \rho C N^{r+1},\qquad \|\omega(t,\cdot)\|_{H^r} \leq \rho C N^{r+2}.
\end{equation}
Thus, for the operator $D$ we have that
\begin{equation}\label{GlobalQuadDuhamControl}
\|D(t,\cdot)\|_{H^r} \leq t  \, \| \omega_0\|_{H^{r+1}} \| u_0\|_{H^{r+1}} \leq C t   \rho^2  N^{2r+5}, \qquad t \geq 0,   
\end{equation}
where we used the properties of the heat kernel and the fact that $H^r$ is an algebra for big values of $r$. The constant $C$ here depends on
$T,r,\nu, \alpha$. 
\\ 
\\
\underline{\em Step 4} \,\,\, Choice of the parameters.\\
\\
We consider the solution at time $t=T$. We rescale the vorticity as
$$
\tilde{\omega}(T,\cdot)=\frac{\omega(T,\cdot)}{\rho},
$$
and from the Duhamel's formula we get that
\begin{equation}\label{FinalFinalmente}
\tilde{\omega}(T,\cdot)-e^{\nu T\Delta} \omega_0^2 = e^{\nu T\Delta}\omega_1+\frac{D(T,\cdot)}{\rho}.
\end{equation}
Recall that 
$$e^{\nu T\Delta} \omega_0^2=\curl\curl(\psi W),$$
has a hyperbolic zero at $x=0$. This implies that if \eqref{FinalFinalmente} is small enough in the $C^1$ norm, the vector field $\omega(T, \cdot)$ has also a hyperbolic zero close to $x=0$ (this just follows applying the implicit function theorem). Thus reconnection must have happened at some time between $t=0$ (where $\omega_0$ has no zeros) and $t=T$ (where $\omega(T, \cdot)$ has at least one hyperbolic zero). 
 
Thus, in order to conclude the proof we are going to show that the $C^1$ norm of \eqref{FinalFinalmente} is small. We will control it with the $H^r$ norm (before we fixed $r=3$, that is enough for this).
By using \eqref{est:finalB1} with $n=2r$ and \eqref{GlobalQuadDuhamControl} we have 
\begin{equation}\label{est:choice-parameters}
\|\tilde{\omega}(T,\cdot)-e^{\nu T\Delta} \omega_0^2\|_{H^r}
\leq 
C(\alpha, T, \nu, r) \left( N^{r+3} e^{-\nu \frac12 N^2 T}
+ \frac{1}{N^{r-4}}  + \rho N^{2r+5} \right).
\end{equation}
Note that, in Step 1 
we have chosen 
$\rho = N^{-\beta}$ with $\beta > 2r + 5$ (and say $r=3$). Thus, 
if we take $N$ sufficiently large with respect to all the other parameters \big(in particular $N \gg \frac{1}{\sqrt{\nu T}}$\big), the right hand side of \eqref{est:choice-parameters} is arbitrarily small, so that the proof is complete. 
\end{proof}

\section{Stability for small times}
In this section we discuss the stability with respect to small times of the reconnection scenario of {\bf Theorem 2}. 
Our first goal is to obtain point-wise decay estimate (in space) of the vorticity $\omega$ associated to the solution $u$ constructed in our {\bf Theorem 2}. These, together with a suitable lower bound on the linear behavior of the vorticity, will allow us to show that the vorticity has no zeros for small times. We need some preliminary lemmas.

\begin{lem}\label{stime decadimento calore}
Let $\alpha \geq 1$ and $\phi$ be the function defined in \eqref{Rec1Size}. Recall that 
$\omega_0=\rho(\omega_0^1+ \omega_0^2)$ with $\omega_0^1$ and
$\omega_0^2$ defined in \eqref{def:omega01}, \eqref{def:omega2}, respectively. 
Then, for $N \geq \alpha$, we have that
\begin{equation}
|e^{\nu t\Delta}\omega_0(x)|\lesssim \rho \frac{N^2  (1 + (\nu t)^{\alpha})}{(1+|x|^2)^{\alpha}}.
\end{equation}
\end{lem}

\begin{proof}
We first prove 
\begin{equation}
|e^{\nu t\Delta}\omega_0^1(x)|\lesssim \frac{N^2  (1 + (\nu t)^{\alpha})}{(1+|x|^2)^{\alpha}}.
\end{equation}
Letting 
$$
\Gamma (x)
:= 
N\nabla\phi \wedge B_{N}+(B_{N}\cdot\nabla)\nabla\phi-\Delta\phi B_{N}-(\nabla\phi\cdot\nabla)B_{N},
$$ 
we have (see \eqref{def:omega01})
$$
\omega_0^1 = N^2\phi B_{N} + \Gamma.
$$
We use the definition of heat kernel and the fact that $|B_N(x)|=1$ to get that
\begin{align}
|e^{\nu t\Delta}\omega_0^1(x)|&=\frac{1}{(4\pi\nu t)^{3/2}}\left|\int_{\R^3}e^{-\frac{|x-y|^2}{4\nu t}}\omega_0^1(y) \de y\right|\nonumber\\
&\leq \frac{N^2}{(4\pi\nu t)^{3/2}}\int_{\R^3}e^{-\frac{|x-y|^2}{4\nu t}}|\phi(y)|\de y\label{stima decadimento 1}\\
&+\frac{1}{(4\pi\nu t)^{3/2}}\int_{\R^3}e^{-\frac{|x-y|^2}{4\nu t}}|\Gamma(y)|\de y\label{stima decadimento 2}.
\end{align}
We start by analyzing \ref{stima decadimento 1}: we use the definition of $\phi$ to write
\begin{align*}
\frac{N^2}{(4\pi\nu t)^{3/2}}\int_{\R^3}e^{-\frac{|x-y|^2}{4\nu t}}|\phi(y)|\de y&\leq \frac{N^2}{(4\pi\nu t)^{3/2}}\int_{\R^3}e^{-\frac{|x-y|^2}{4\nu t}}\frac{1}{(1+|y|^2)^{\alpha}}\de y\\
&=\frac{N^2}{(4\pi\nu t)^{3/2}}\int_{\R^3}e^{-\frac{|z|^2}{4\nu t}}\frac{1}{(1+|x-z|^2)^{\alpha}}\de z\\
&=\frac{N^2}{(4\pi\nu t)^{3/2}}\frac{1}{(1+|x|^2)^{\alpha}}\int_{\R^3}e^{-\frac{|z|^2}{4\nu t}}
\frac{(1+|x|^2)^{\alpha}}{(1+|x-z|^2)^{\alpha}}\de z\\
&\lesssim \frac{N^2}{(1+|x|^2)^{\alpha}} \frac{1}{(4\pi\nu t)^{3/2}} \int_{\R^3}e^{-\frac{|z|^2}{4\nu t}} (1+|z|^2)^{\alpha} \de z 
\\
&\lesssim \frac{N^2}{(1+|x|^2)^{\alpha}}(1+(\nu t)^{\alpha}),
\end{align*}
where we used 
$$
\frac{(1+|x|^2)^{\alpha}}{(1+|x-z|^2)^{\alpha}} \lesssim (1+|z|^2)^{\alpha}, \qquad \alpha \geq 1.
$$
The analogous estimate on \eqref{stima decadimento 2} follows in a similar way using that 
$\partial_j \phi$ and  $\partial_{ij} \phi$ decay faster than $\phi$ and that their modulus is bounded by $C \alpha$ 
and $C \alpha^2$, respectively. 

Then, in order to conclude the proof, we must prove 
\begin{equation}
|e^{\nu t\Delta}\omega_0^2(x)|\lesssim \frac{N^2  (1 + (\nu t)^{\alpha})}{(1+|x|^2)^{\alpha}}.
\end{equation} 
We first consider $|x|\leq 1$. In this case, the bound essentially 
follows by Lemma \ref{lem:stima-omega02}, that in fact gives a much better 
estimate (for $N = N(\nu, T) \gg1$). Indeed, recalling 
the definition \eqref{def:omega2} of $\omega_0^2$
and noticing that 
$$
e^{\nu t \Delta} \omega_0^2  = e^{-\nu (T - t) \Delta}\curl\curl(\psi W),
$$ 
we 
easily get the desired bound proceeding as in the proof of Lemma \ref{lem:stima-omega02}.
Moreover, arguing as in \eqref{stima-omega02}, one can show that for $|x|>1$,
$$
|e^{\nu t\Delta}\omega_0^2|\leq \frac{C(T,\nu)}{(1+|x|^2)^\alpha},
$$
and again here we could get an even better bound, namely for arbitrarily
large values of $\alpha$. This concludes the proof.
\end{proof}

We now define the functional space
\begin{equation}%\label{def:Linftygamma}
L^\infty_\gamma(\R^3):=\left\{ f \,\,\mbox{measurable }:\mathrm{ess}\sup_{x\in\R^3}(1+|x|)^\gamma|f(x)|<\infty\right\}.
\end{equation}
In this section we are interested in the very small time behavior of 
the vorticity, and more precise $t \lesssim \frac{1}{\nu N^2}$, with $N \gg 1$. This justify the 
restriction to times $t \leq \frac{1}{\nu}$ in the following lemma, in which we prove that the vorticity $\omega$ constructed in {\bf Theorem 2} belongs to the space $C([0,\nu^{-1}];L^\infty_{2 \alpha}(\R^3))$ if $N$ is big enough.

\begin{lem}\label{lem:esistenza spazi pesati 2}
Let $r>5/2$ and $u$ be the solution of \eqref{eq:ns} constructed in {\bf Theorem 2} (in particular, under the 
choice $\rho = N^{-\beta}$ with $\beta=4r>2r+5$).  Then for the relative vorticity $\omega$ and for times 
for $t \leq 1/\nu$ we have that
\begin{equation}\label{est:decadimento omega}
|\omega(t,x) - e^{\nu t \Delta} \omega_0(x)| \leq 
  \sqrt{\frac{t}{\nu}}   \frac{\rho N^{2 - 2r }}{(1+|x|^2)^{\alpha}} ,
\end{equation}
for all $N=N(r,T,\alpha,\nu)$ sufficiently large.
\end{lem}
\begin{proof}
First, we use Lemma \ref{stime decadimento calore}  to show 
\begin{equation}\label{Comb22222}
\sup_{s\in [0,t]}\sup_{x\in\R^3}(1+|x|^2)^{\alpha}|e^{\nu s\Delta}\omega_0|\lesssim 
\rho N^2; 
\end{equation}
remember that here we have restricted to $t \leq 1/\nu$.  Then, we note 
that by \eqref{eq:InitiHr} we have
\begin{equation}\label{Comb11111}
\| e^{\nu s\Delta}\omega_0 \|_{H^r} \leq 
\| \omega_0 \|_{H^r} \leq C(T,r,\nu, \alpha) \rho   N^{r+2}.
\end{equation}

To prove the statement we notice that, once $u$ is given, we can recover $\omega$ solving the integral problem 
%Let $u$ be given by Lemma \ref{lem:esistenza spazi pesati} and consider the operator
\begin{equation}\label{def:Phi}
\Phi[\omega]=e^{\nu t\Delta}\omega_0-\int_0^t  e^{\nu (t-s)\Delta} \nabla \cdot\left(u(s,\cdot)\otimes\omega(s,\cdot)-\omega(s,\cdot)\otimes u(s,\cdot)\right)\de s,
\end{equation}
by a fixed point argument in the functional space (the intersection is endowed with the sum of the norms)
$$
Z := C([0,\nu^{-1}]; H^{r-1}(\R^3)) \cap  
C([0,\nu^{-1}];L^\infty_{2 \alpha}(\R^3)), 
$$
where we restrict to $r > 5/2$, to ensure the embedding $H^{r-1} \hookrightarrow L^{\infty}$.
By uniqueness of $H^{r-1}$ solutions this is sufficient to show that 
$\omega$ is indeed the vorticity associated to the solution $u$.\\
Combining \eqref{Comb22222} and \eqref{Comb11111}
we get (remember that $\beta =4r$)
\begin{align}
\| e^{\nu s\Delta}\omega_0 \|_{Z} &\leq C(T,r,\nu, \alpha) \rho(N^2+   N^{r+2})\nonumber \\
&\leq C(T,r,\nu, \alpha)  (N^{2-4r}+ N^{r+2 - \beta}) \nonumber\\
&\leq C(T,r,\nu, \alpha)   N^{2 - 3r } \label{ContrRadius},
\end{align}
taking $N = N(T,r,\nu, \alpha)$ sufficiently large.
Thus,
to solve \eqref{def:Phi} in a ball of radius $N^{2-3r}$ it is sufficient to prove the contractive property
$$
\| \Phi[\omega_1] - \Phi[\omega_2]  \|_{Z} \leq \theta \| \omega_1 - \omega_2\|_{Z},
\qquad \theta <1. 
$$
Hence, by using \eqref{HeatFlowBound3}, the algebra property of $H^{r-1}$, and \eqref{eq:InitiHr} we have
\begin{align}
\| \Phi[\omega_1](t,\cdot)-\Phi[\omega_2](t,\cdot) \|_{H^{r-1}}  
& \lesssim 
\|u\|_{L^{\infty}((0,T); H^{r-1})} \|\omega_1-\omega_2\|_{C([0,T];H^{r-1})} \frac{1}{\sqrt{\nu}} 
\int_{0}^{t} (t-s)^{-1/2} \, ds 
\nonumber\\
& \lesssim 
2 \rho C(T,r,\nu, \alpha)  N^{r}  \sqrt{\frac{t}{\nu}} \|\omega_1-\omega_2\|_{C([0,T];H^r)}.
 \end{align}
Then, recalling that $\rho = N^{-\beta}$ and $t \nu \leq 1$
we also have that
$$
2 \rho C(T,r,\nu, \alpha)  N^{r}  \sqrt{\frac{t}{\nu}} \ll  1,
$$ 
since $\beta= 4r$ and we take  $N = N(r, \alpha, T, \nu)$ sufficiently large.
Thus it remains  to prove 
$$
\| \Phi[\omega_1] - \Phi[\omega_2]  \|_{L^{\infty}([0,T];L^\infty_{2 \alpha}(\R^3))} \leq \theta \| \omega_1 - \omega_2\|_{Z}. 
$$ 
Let 
$$
F(x,t) :=   \frac{|x|}{\nu t}  e^{-\frac{|x|^2}{4 \nu t}}. 
$$
 We have that
\begin{align*}
& |\Phi[\omega_1](t,x)-\Phi[\omega_2](t,x)| \lesssim \int_0^t \int_{\R^3} F(t-s,x-y) |u(s,y)||\omega_1(s,y)-\omega_2(s,y)|\de y \de s\\
& \leq \|u\|_{L^{\infty}((0,T); H^{r-1})} \|\omega_1-\omega_2\|_{C([0,T];L^\infty_{2\alpha}(\R^3))} 
\int_0^t \int_{\R^3}\frac{F(t-s,x-y)}{(1+|y|^2)^{\alpha}}\de y \de s.
%\\
%&\leq \bar{C}(T,N)\|\omega_1-\omega_2\|_{C([0,T];L^\infty_{2\alpha}(\R^3))}\int_0^t \int_{\R^3}\frac{|F(t-s,x-y)|}{(1+|y|^2)^2}\de y\de s,
\end{align*}
Now, we first consider $|x|\leq 1$. In this case we simply use %  and we obtain that 
\begin{align}
%\sup_{|x|\leq 1} 
\int_0^t \int_{\R^3}\frac{F(t-s,x-y)}{(1+|y|^2)^{\alpha}}\de y \de s
&\leq  \int_0^t \int_{\R^3} F(t-s,x-y) \de y\de s\nonumber\\
&\leq \frac{1}{\sqrt{\nu}} \int_0^t(t-s)^{-1/2}\de s \lesssim \sqrt{\frac{t}{\nu}}.
%&\leq \sqrt{T}\bar{C}(T,N)\|\omega_1-\omega_2\|_{C([0,T];L^\infty_2(\R^3))}
\end{align}
For $|x|>1$, instead, we use %use \eqref{stime decadimentoF} to obtain
\begin{equation}
\sup_{|x|>1}\int_0^t \int_{\R^3}\frac{F(t-s,x-y)}{(1+|y|^2)^{\alpha}}\de y \de s \lesssim 
\sup_{|x|> 1} \int_0^t \int_{\R^3}\frac{e^{- \frac{|x-y|^2}{4\nu (t-s)} }}{(\nu (t-s))^{3/2}} \frac{|x-y|}{\nu (t-s)} \frac{1}{(1+|y|^2)^{\alpha}}\de y \de s. 
\end{equation}
We split the integral in two regions: for $|x-y|<|x|/2$ we have $|y| \geq \frac{|x|}{2}$ and thus 
\begin{align*}
\int_0^t & \int_{|x-y|<|x|/2} 
\frac{e^{- \frac{|x-y|^2}{4\nu (t-s)} }}{(\nu (t-s))^{3/2}} \frac{|x-y|}{\nu (t-s)} \frac{1}{(1+|y|^2)^{\alpha}}\de y \de s
 \\ 
&\lesssim \frac{1}{(1+|x|^2)^{\alpha}}  \int_0^t \int_{\R^3} \frac{e^{- \frac{|x-y|^2}{4\nu (t-s)} }}{(\nu (t-s))^{3/2}} \frac{|x-y|}{\nu (t-s)} \de y \de s
\lesssim   \sqrt{\frac{t}{\nu}}\frac{1}{(1+|x|^2)^{\alpha}}.
\end{align*}
For the second region, namely $|x-y|>|x|/2$, we have
\begin{align*}
\int_0^t & \int_{|x-y|<|x|/2}\frac{e^{- \frac{|x-y|^2}{4\nu (t-s)} }}{(\nu (t-s))^{3/2}} \frac{|x-y|}{\nu (t-s)}  \frac{1}{(1+|y|^2)^{\alpha}} \de y \de s
%& \lesssim 
%  \int_0^t \int_{\R^3}\frac{e^{- \frac{|x|^2}{8 \nu (t-s)} }}{(4 \pi \nu t)^{3/2}} \frac{|x-y| }{(1+|y|^2)^{\alpha}} \de y \de s
\\ 
& \lesssim 
e^{- \frac{|x|^2}{16 \nu t} }
 \int_0^t \int_{\R^3} \frac{e^{- \frac{|x-y|^2}{8\nu (t-s)} }}{(\nu (t-s))^{3/2}} \frac{|x-y|}{\nu (t-s)}   \de y \de s
\lesssim 
\sqrt{\frac{t}{\nu}} e^{- \frac{|x|^2}{16 \nu t} } 
\lesssim  \sqrt{\frac{t}{\nu}}   \frac{1}{(1+|x|^2)^{\alpha}},
\end{align*}
where in the last inequality we have used $\nu t \leq 1$. Finally, by collecting these estimates we have obtained that
\begin{align}\nonumber
\sup_{x \in \mathbb{R}^3}  (1+|x|^2)^{\alpha} |\Phi[\omega_1](t,x)-\Phi[\omega_2](t,x)| 
& 
\lesssim \sqrt{\frac{t}{\nu}}
 \|u\|_{L^{\infty}((0,T); H^{r-1})} \|\omega_1-\omega_2\|_{C([0,T];L^\infty_{2\alpha}(\R^3))} 
\\ 
 & \lesssim  \label{CompareWitThis}
 \sqrt{\frac{t}{\nu}} \rho N^{r} 
C(r, \alpha, T ,\nu)  \|\omega_1-\omega_2\|_{C([0,T];L^\infty_{2\alpha}(\R^3))},
\end{align}
where we used \eqref{GlobalControl} in the last inequality.
Recalling that $\rho=N^{-\beta}$ and that $t \nu \leq 1$ we have
$$
 \sqrt{\frac{t}{\nu}}
C(r, \alpha, T, \nu) N^{r  - \beta} \ll 1,
$$ 
since $\beta = 4r$ and taking $N = N(r, \alpha, T, \nu)$ sufficiently large. 
This allow to close the fixed point argument and it shows that \eqref{def:Phi} has a unique solution in the ball of radius $N^{2-3r}$, in the topology induced by the $Z$-norm. The size of the radius of the ball was prescribed by inequality \eqref{ContrRadius}. By uniqueness in $H^{r-1}$, this solution coincides with the vorticity $\omega = \curl u$. Thus 
\begin{equation}\label{def:PhiDOPO}
\omega(\cdot, t) - e^{\nu t\Delta}\omega_0= -\int_0^t  e^{\nu (t-s)\Delta} \nabla \cdot\left(u(s,\cdot)\otimes\omega(s,\cdot)-\omega(s,\cdot)\otimes u(s,\cdot)\right)\de s,
\end{equation}
and proceeding as before we also get the estimate (compare with \eqref{CompareWitThis})
\begin{align*}
\| \omega(\cdot, t) & - e^{\nu t\Delta}\omega_0 \|_{L^{\infty}([0,T];L^\infty_{2 \alpha}(\R^3))}
 \leq
\sqrt{\frac{t}{\nu}} \rho N^{r} 
C(r, \alpha, T ,\nu) \| \omega \|_{C([0,T];L^\infty_{2\alpha}(\R^3))}
\\ \nonumber
& \leq
\sqrt{\frac{t}{\nu}}
C(r, \alpha, T ,\nu) \rho N^{r +2 - 3r} 
\leq
\sqrt{\frac{t}{\nu}} 
C(r, \alpha, T ,\nu) \rho N^{2 - 2r},
\end{align*}
taking $N = N(r, \alpha, T, \nu)$ sufficiently large,
from which the inequality \eqref{est:decadimento omega} follows.  
\end{proof}

Now we are ready to prove that the reconnection in {\bf Theorem 2} is not instantaneous and it must happen on a resistive time scale, in particular for times proportional to $\nu^{-1}N^{-2}$. 
The result is the following.
\begin{thm}\label{thm:tempo}
Let $u$ be the solution of \eqref{eq:ns} constructed in {\bf Theorem 2}.  Then, there exists a constant $C^*>0$, with $C^*\lesssim\frac{1}{\nu N^{2}}$, such that $\omega(t,x)$ does not have zeros for any $t\in [0,C^*)$. In particular, the reconnection must occur for times of the order of $\mathcal{O}(\nu^{-1}N^{-2})$.
\end{thm}
\begin{proof}
First of all, notice that
$$
|\omega(t,x)| \geq |e^{\nu t \Delta} \omega_0 (x)| - |\omega(t,x) - e^{ \nu t \Delta} \omega_0(x)|,
$$
thus using inequality \eqref{est:decadimento omega}, the statement follows once we have proved that 
\begin{equation}\label{dnsgkoldskjgksl}
|e^{\nu t \Delta} \omega_0 (x)| \geq 2  \sqrt{\frac{t}{\nu}} \frac{\rho N^{2-2r}  }{(1+|x|^2)^{\alpha}}.
\end{equation}
We recall \eqref{VorticityFinalDef}, namely that  
$$
\omega_0 := \rho (\omega_0^1 + \omega_0^2).
$$ 
By using \eqref{est:lb-omega1} we have that
\begin{align}
|e^{\nu t\Delta}\omega_0 (x)|&\geq \rho|\omega_0^1 (x)|- \rho|e^{\nu t\Delta}\omega_0^1 (x)-\omega_0^1 (x)| -\rho
|e^{\nu t\Delta}\omega_0^2 (x)| \nonumber\\
& \geq {\rho}\left(\frac{c N^2}{(1+|x|^2)^{\alpha}}-  \underbrace{|e^{\nu t\Delta}\omega_0^1 (x)-\omega_0^1 (x)|}_{(*)}
- |e^{\nu t\Delta}\omega_0^2 (x)|\right).
\label{costante}
\end{align}
%Since $cN^2 \gg 1$, in order to conclude the proof, we show that for $t$ small 
%
Since, for times say $t \leq 1$ and $N=N(\nu, r) \gg1$, we have that 
$c {\rho}N^2 \gg \sqrt{\frac{t}{\nu}} \rho  N^{2-2r}$, the first term $\frac{c \rho N^2}{(1+|x|^2)^{\alpha}}$ 
of \eqref{costante} satisfies  the lower bound  
\eqref{dnsgkoldskjgksl} (in fact, with a much larger constant). 
Thus %, taking $N$ sufficienlty large, 
it is enough to prove that the remaining quantities inside the parenthesis are smaller than $\frac{cN^2}{(1+|x|^2)^{\alpha}}$. 
This is clearly true for %The same is true for the third term 
$|e^{\nu t\Delta}\omega_0^2 (x)|$, that satisfy en even better point-wise bound, as we have seen in  
the proof of Lemma \ref{stime decadimento calore}.\\
\\
It thus remains to show that $(*)$ is smaller than $\frac{cN^2}{(1+|x|^2)^{\alpha}}$. By using the explicit formula for the heat kernel, we have that 
\begin{equation*}
|e^{\nu t\Delta}\omega_0^1 (x)-\omega_0^1 (x)| 
\leq \frac{1}{(4\pi\nu t)^{3/2}}\int_{\R^3}e^{-\frac{y^2}{4\nu t}}|\omega_0^1(x-y)-\omega_0^1(x)|\de y.
\end{equation*}
We recall that $\omega_0^1$ is defined as
$$
\omega_0^1=N^2\phi B_{N}+ N\nabla\phi \wedge B_{N}+(B_{N}\cdot\nabla)\nabla\phi-\Delta\phi B_{N}-(\nabla\phi\cdot\nabla)B_{N}.
$$
We do an explicit computation for the part involving the main contribution, i.e. $N^2\phi B_{N}$. It will be clear from the proof that the remaining terms can be treated in an analogous way. In fact they decay even better since the partial derivatives of $\phi$ decays better than $\phi$.\\
We compute
\begin{align}
\frac{N^2}{(4\pi\nu t)^{3/2}}\int_{\R^3}e^{-\frac{y^2}{4\nu t}}&|\phi(x-y) B_{N}(x-y)-\phi(x) B_{N}(x)|\de y
\nonumber
\\
&\leq \frac{N^2}{(4\pi\nu t)^{3/2}}\int_{\R^3}e^{-\frac{y^2}{4\nu t}} |\phi(x-y)-\phi(x)||B_N(x-y)|\de y \label{split1} 
\\
&+\frac{N^2}{(4\pi\nu t)^{3/2}}\int_{\R^3}e^{-\frac{y^2}{4\nu t}}|\phi(x)||B_{N}(x-y)- B_{N}(x)|\de y \label{split2} .
\end{align}
Since $|\nabla B_N(z)| \lesssim N$ (note that the bound is uniform over $z \in \R^3$), we get
$$
|B_N(x-y) - B_N(x)|  %= \left| \int_0^1 \frac{d}{ds} B_N(x-sy) \, \ds \right|
= \left| \int_0^1 y \cdot \nabla B_N(x-sy) \, ds \right| \leq |y| \sup_{z \in \mathbb{R}^3} |\nabla B_N(z)| \leq |y| N.
$$ 
Similarly
$$
|\phi(x-y) - \phi(x)|  %= \left| \int_0^1 \frac{d}{ds} B_N(x-sy) \, \ds \right|
= \left| \int_0^1 y \cdot \nabla \phi (x-sy) \, ds \right| \leq |y| \sup_{s \in [0,1], \, y \in \R^3}  | \nabla \phi (x- s y)| 
\leq C \alpha (|y| + |y|^{2\alpha + 1}) |\phi(x)|,
$$
where we used Lemma \ref{ForthLemmaS7} in the last inequality
that is valid for some constant $C(\alpha)$ that is independent on $z$. 
Thus
for \eqref{split1} we have that (recall $\nu t \leq 1$)
\begin{align*}
|\eqref{split1}| & \lesssim  \frac{ |\alpha| N^2}{(4\pi\nu t)^{3/2}}|\phi(x)| \int_{\R^3}e^{-\frac{y^2}{4\nu t}}(|y| + |y|^{2\alpha+1})   \de y 
\\ \nonumber
&\lesssim   %\int_{\R^3}e^{-\frac{y^2}{4\nu t}}|y|^2\de y\\
\frac{|\alpha| N^2}{(1 + |x|)^{\alpha}} \frac{1}{(4\pi\nu t)^{3/2}}  \int_{\R^3}e^{-\frac{y^2}{4\nu t}}(|y| + |y|^{2\alpha+1})  \de y 
\lesssim \frac{N^2 (\sqrt{\nu t} + (\nu t)^{\frac{2\alpha +1}{2}}  )}{(1 + |x|)^{\alpha}} 
\lesssim \frac{N^2 \sqrt{\nu t}}{(1 + |x|)^{\alpha}} .
\end{align*} 
The main contribution is that of \eqref{split2}, that is bounded as
\begin{equation*}
|\eqref{split2}|  \lesssim \frac{N^3}{(1 + |x|)^{\alpha}} \frac{1}{(4\pi\nu t)^{3/2}} \int_{\R^3}e^{-\frac{y^2}{4\nu t}}|y|   \de y 
\lesssim   %\int_{\R^3}e^{-\frac{y^2}{4\nu t}}|y|^2\de y\\
 \frac{N^3 \sqrt{\nu t}}{(1 + |x|)^{\alpha}} .
\end{equation*}
These are smaller than than $\frac{cN^2}{(1+|x|^2)^{\alpha}}$
precisely for 
\begin{equation}
t\leq \frac{c^*}{\nu N^2},
\end{equation}
where $c^*>0$ is a suitable small constant.
Thus, the solution $\omega$ does have zeros for $t\leq \frac{c^*}{\nu N^2}$. On the other hand, the quantity in \eqref{est:choice-parameters} can be made arbitrary small provided that 
$$
\nu N^2 T\gg 1\Longrightarrow  T\gg \frac{1}{\nu N^2}.
$$
This implies that the reconnection must occur for a time of the order $\frac{c^{**}}{\nu N^2}$, with $c^{**} \gg c^*$. This concludes the proof.
\end{proof}

\begin{rem}\label{rem:tempo locale}
It is well-known that any $u_0\in H^r(\R^3)$, with $r>\frac52$, gives rise to a local smooth solution on some interval $(0,T^*)$ with the local time of existence satisfying 
$$
T^*\geq C \max\left\{ \frac{1}{\|u_0\|_{H^r}}, \frac{\nu^3}{\|u_0\|^4_{H^1}} \right\} .
$$
In particular, the maximal time of existence has a positive, $\nu$-independent
lower bound (notice that we have enough regularity to locally solve the Euler system).
This prevents our strategy from working with local smooth solutions, without providing a longer existence time for the initial datum $u_0$ in \eqref{def:u_0}. In fact,  
if we set $\rho=1$ in \eqref{def:u_0} we would have a local smooth solution defined up to time $T^*\sim \max\left\{ \frac{1}{N^4}, \frac{\nu^3}{N^{8}}\right\}$, which is much smaller than the reconnection time which is of order $\sim \nu^{-1}N^{-2}$ if $\nu$ is small (recall that in the proof of {\bf Theorem 2} we set $r=3$). However, this suggest that our proof could work with local smooth solutions if we are allowed to choose large values of the viscosity $\nu$ depending on $N$, for example $\nu\gg N^{3/2}$.
\end{rem}

\appendix
\section{}

We provide here two technical lemmas we used in the proof of Theorem \ref{thm:tempo} and {\bf Theorem 2} respectively.
\begin{lem}\label{ForthLemmaS7}
Let $\alpha \geq 1$ and let $\phi(x) := \frac{1}{(1 + |x|^2)^{\alpha}}$. Then
\begin{equation}\label{SupCalc}
\sup_{s \in [0,1], \, y \in \R^3}  |\nabla \phi(x- s y)| 
\leq C \alpha (1+ |y|^{2\alpha}) |\phi(x)|.
\end{equation}
\end{lem}
\begin{proof}
First of all, we have that
\begin{equation}\label{NablaExpansionOfPhi}
\nabla \phi(x- s y) = - \frac{2 \alpha (x-sy)}{(1 + |x- s y|^2)^{\alpha +1}}.
\end{equation} 
When we compute the $\sup$ in \eqref{SupCalc} we distinguish two regimes. 
First we consider $s,y$ such that 
$|x- s y| \geq \frac12 |x|$. In this case we simply bound $\frac{1}{(1 + |x- s y|^2)^{\alpha}} \leq \frac{1}{(1 + |x|^2)^{\alpha}}$, 
that plugged into \eqref{NablaExpansionOfPhi} gives the \eqref{SupCalc} and, more 
precisely, the $\lesssim \alpha |\phi(x)| $ on the right hand side. 
If we rather consider $s,y$ such that 
$|x- s y| < \frac12 |x|$ it means that $|x| \leq 2 |sy| \leq 2 |y|$ (recall $s \in [0,1]$), obtaining that 
\begin{align*}
|\nabla \phi(x- s y)|&=2\alpha \frac{ |x-sy|}{(1 + |x- s y|^2)^{\alpha +1}}\\
&\leq \alpha|\phi(x)|\frac{ |x- s y|}{(1 + |x- s y|^2)^{\alpha +1}}(1+|x|^2)^\alpha\\
&\leq C\alpha|\phi(x)|(1+|x|^2)^\alpha\\
&\leq C\alpha|\phi(x)|(1+|y|^{2\alpha}),
\end{align*}
which thus implies \eqref{SupCalc}.
\end{proof}

We now prove the quantitative statement \eqref{eq:stima Hs quantitativa}. %of Theorem \ref{lem:stima-ns} 
\begin{lem}\label{stima norme ns}
Let $r \geq 1$ integer. Let $\nu>0$ and $u_0 \in H^r(\R^3)$ be a divergence-free vector field such that
\begin{equation}
\|u_0\|_{H^k}\leq C N^{k+1-\beta},
\end{equation}
for all $k=0,...,r$, with $\beta>2$. For all  
$N = N(\nu, k, \beta)$ sufficiently large the following holds. There exists a unique global strong solution 
$$
u\in C(\R^+; H^r(\R^3))\cap L^2(\R^+;\dot H^{r+1}(\R^3)),
$$ 
such that, for all $k=0,..., r$, the following bound holds
\begin{equation}
\sup_{t\in[0,\infty)}\|u(t,\cdot)\|_{H^k}^2+ \nu \int_0^\infty \|\nabla u(s,\cdot)\|_{H^k}^2\de s\leq C_k\|u_0\|_{H^k}^2,
\end{equation}
for some positive constant $C_k>0$.
\end{lem}
\begin{proof}
Note that if $u$ solves the Navier-Stokes equation with viscosity $\nu = 1$, then 
$u^{\nu} (x,t) := \nu u (x, \nu t)$ solves the Navier-Stokes equation with viscosity $=\nu$. This allows us to reduce to prove the 
theorem setting $\nu =1$.
First of all, notice that 
\begin{equation}
\|u_0\|_{{H}^{\frac12}}\leq\|u_0\|_{L^2}^\frac12\|\nabla u_0\|_{L^2}^\frac12\ll 1,
\end{equation}
thus Theorem \ref{lem:stima-ns} implies that there exists a unique global smooth solution $u\in C(\R^+; H^r(\R^3))\cap L^2(\R^+;\dot{H}^{r+1}(\R^3))$ starting from $u_0$. We now prove the quantitative estimate. 
For $k=0$ the statement follows by the energy identity. Thus we assume $k \geq 1$.
\\
\\
\emph{\underline{Proof for $k=1$}} %{\color{blue} Come Lemma 6.13 di \cite{RRS}}.\\
\\
\\
In this step we use the smallness assumption on the initial datum. We differentiate the equations \eqref{eq:ns} and we multiply by $\nabla u$ we obtain that
\begin{align*}
\frac12\frac{\de}{\de t}\|\nabla u\|^2_{L^2}+\|\nabla^2 u\|^2_{L^2}&\leq \int_{\R^3}|u||\nabla u||\nabla^2 u|\de x\\
&\leq \|u\|_{L^6}\|\nabla u\|_{L^3}\|\nabla^2 u\|_{L^2}\\
&\leq \|\nabla u\|_{L^2}\|\nabla u\|_{L^2}^{\frac12}\|\nabla u\|_{L^6}^{\frac12}\|\nabla^2 u\|_{L^2}.
\end{align*}
Then we use that
$$
\|\nabla u\|_{L^6}\leq C\|\nabla^2 u\|_{L^2},
$$
to compute
\begin{align*}
\frac12\frac{\de}{\de t}\|\nabla u\|^2_{L^2}+\|\nabla^2 u\|^2_{L^2}&\leq C\|\nabla u\|_{L^2}^\frac32\|\nabla^2 u\|_{L^2}^\frac32.
\end{align*}
We use the Gagliardo–Nirenberg inequality
$$
\|\nabla u\|_{L^2}\leq C\|u\|_{L^2}^\frac12\|\nabla^2 u\|_{L^2}^\frac12,
$$
and together with Young's inequality with exponents $8$ and $\frac87$ we deduce that
\begin{align*}
\|\nabla u\|_{L^2}^\frac32\|\nabla^2 u\|_{L^2}^\frac32&\leq C\|u\|_{L^2}^\frac14\|\nabla^2 u\|_{L^2}^\frac14\|\nabla u\|_{L^2}\|\nabla^2 u\|_{L^2}^\frac32\\
&\leq C\|u\|_{L^2}^2\|\nabla u\|_{L^2}^8+\frac12\|\nabla^2 u\|_{L^2}^2\\
&\leq C \|u\|_{L^2}^4\|\nabla u\|_{L^2}^4\|\nabla^2 u\|_{L^2}^2+\frac12\|\nabla^2 u\|_{L^2}^2.
\end{align*}
Hence, we have obtained that
\begin{equation}\label{tempo infinito k=1}
\frac{\de}{\de t}\|\nabla u\|_{L^2}^2\leq \|\nabla^2 u\|_{L^2}(C \|u\|_{L^2}^4\|\nabla u\|_{L^2}^4-1).
\end{equation}
Moreover, from the energy inequality we know that
\begin{equation}
\frac{\de}{\de t}\| u\|_{L^2}^2=-\|\nabla u\|_{L^2}^2.
\end{equation}
Then, it is immediate to obtain that
\begin{align*}
\frac{\de}{\de t}(\| u\|_{L^2}^2\|\nabla u\|_{L^2}^2)\leq \| u\|_{L^2}^2\|\nabla^2 u\|_{L^2}(C \|u\|_{L^2}^4\|\nabla u\|_{L^2}^4-1)-\|\nabla u\|_{L^2}^4,
\end{align*}
and so, if $\|u_0\|_{L^2}^4\|\nabla u_0\|_{L^2}^4<\frac1C$, a classical continuity argument implies that $\| u(t,\cdot)\|_{L^2}^2\|\nabla u(t,\cdot)\|_{L^2}^2$ is a decreasing function. Then, it follows that $\|\nabla u(t,\cdot)\|_{L^2}^2$ is decreasing as well, obtaining that
$$
\|\nabla u(t,\cdot)\|_{L^2}^2\leq \|\nabla u_0\|_{L^2}^2.
$$
Moreover, the above argument also implies that 
$$
C\| u(t,\cdot)\|_{L^2}^4\|\nabla u(t,\cdot)\|_{L^2}^4-1\leq C\|u_0\|_{L^2}^4\|\nabla u_0\|_{L^2}^4-1<\tilde C<1,
$$
and substituting in \eqref{tempo infinito k=1} we also get that
\begin{equation}
\int_0^\infty\|\nabla^2 u\|_{L^2}^2\de s\leq  \frac{1}{\tilde C}\|\nabla u_0\|_{L^2}^2,
\end{equation}
which implies the claimed bound with $C_1=1+\frac{1}{\tilde C}$.\\
\\
\emph{\underline{Proof for every $k\geq 2$}}\\
\\
We take the $H^k$ inner product of the equation with $u$ and we get that
\begin{align*}
\frac12\frac{\de}{\de t}\|u\|_{H^k}^2+\|\nabla u\|_{H^k}^2&\leq |<(u\cdot\nabla)u,u>_{H^k}|\\
&=  |<u\otimes u,\nabla u>_{H^k}|\\
&\leq \|u\otimes u\|_{H^k}\|\nabla u\|_{H^k}\\
&\leq 2\|u\|_{L^\infty}\|u\|_{H^k}\|\nabla u\|_{H^k}\\
&\leq C\|u\|_{L^\infty}^2\|u\|_{H^k}^2+\frac12 \|\nabla u\|_{H^k}^2,
\end{align*}
where in the third line we used the tame estimate
$$
\|uv\|_{H^k}\lesssim \|u\|_{L^\infty}\|v\|_{H^k}+\|v\|_{L^\infty}\|u\|_{H^k},
$$
which holds for $k>\frac32$. Then, if we absorb the gradient term on the left hand side we get that
\begin{equation}
\frac{\de}{\de t}\|u\|_{H^k}^2+\|\nabla u\|_{H^k}^2\leq C\|u\|_{L^\infty}^2\|u\|_{H^k}^2,
\end{equation}
and by Gronwall's inequality we obtain that
\begin{align}
\|u(t,\cdot)\|_{H^k}^2&\leq \|u_0\|_{H^k}^2e^{C\int_0^t \|u\|_{L^\infty}^2\de s}\nonumber\\
&\leq \|u_0\|_{H^k}^2 e^{C\int_0^t\|u\|_{H^2}^2\de s}\nonumber\\
&\leq \|u_0\|_{H^k}^2e^{CC_1\|u_0\|_{H^1}^2},\label{ammo finito}
\end{align}
where in the second line we used the Sobolev embedding and in the last line we used the estimate for $k=1$. Then, we use \eqref{ammo finito} to obtain that
\begin{align}
\int_0^\infty \|\nabla u(s,\cdot)\|_{H^{k}}^2\de s&\leq\|u_0\|_{H^k}^2+C\int_0^\infty\|u(s,\cdot)\|_{L^\infty}^2\|u(s,\cdot)\|_{H^k}^2\de s\nonumber\\
&\leq \left(1+Ce^{CC_1\|u_0\|_{H^1}^2}\int_0^\infty\|u(s,\cdot)\|_{H^2}^2\de s\right)\|u_0\|_{H^k}^2\nonumber\\
&\leq \left(1+CC_1e^{CC_1\|u_0\|_{H^1}^2}\right)\|u_0\|_{H^k}^2.
\end{align}
Since $\beta>2$, we have that
$$
\|u_0\|_{H^1}\leq CN^{2-\beta}\ll 1,
$$
which finally leads to
\begin{equation}
\| u(t,\cdot)\|_{H^{k}}^2+\int_0^\infty \|\nabla u(s,\cdot)\|_{H^{k}}^2\de s\leq C_k\|u_0\|_{H^k}^2,
\end{equation}
where, for all $k\geq 2$
$$
C_k=1+[1+CC_1]e^{CC_1\|u_0\|_{H^1}^2}.
$$
Note that, as long as $\beta>2$, the constant $C_k$ can be chosen to be independent of the norms of the initial datum. This concludes the proof. 
\end{proof}

\section*{Conflict of interest}

There is no conflict of interest to declare.

\section*{Data availability}

The manuscript has no associated data.

\section*{Acknowledgements}
During the preparation of this manuscript GC has been supported by the ERC STARTING GRANT 2021 ``Hamiltonian Dynamics, Normal Forms and Water Waves" (HamDyWWa), Project Number: 101039762. Views and opinions expressed are however those of the authors only and do not necessarily reflect those of the European Union or the European Research Council. Neither the European Union nor the granting authority can be held responsible for them. GC is partially supported by INdAM-GNAMPA, by the projects PRIN 2020 ``Nonlinear evolution PDEs, fluid dynamics and transport equations: theoretical foundations and applications”, and PRIN2022 ``Classical equations of compressible fluids mechanics: existence and properties of non-classical solutions''. RL is supported by the Basque Government under program BCAM-BERC 2022-2025 and by the Spanish Ministry of Science, Innovation and Universities through the BCAM Severo Ochoa accreditation CEX2021-001142-S and the project PID2021-123034NB-I00. RL is also supported by the Ramon y Cajal fellowship RYC2021-031981-I. The authors are grateful to Daniel Peralta-Salas for useful conversations on the topics of this paper and to Claudia Peña Vázquez de la Torre for carefully reading the first version of the manuscript. The authors thanks the anonymous referees for numerous suggestions and comments, which have improved this paper.

\end{document}